\documentclass[11pt,a4paper]{article}
\usepackage{graphicx}
\usepackage{amsmath,amsfonts,amsthm}
\usepackage{amssymb}
\usepackage{multirow,array}
\usepackage{comment}
\usepackage{enumitem}
\usepackage[utf8]{inputenc}
\usepackage[T1]{fontenc}
\usepackage{emptypage}
\usepackage{hyperref}
\usepackage{color}
\usepackage{mathrsfs}
\usepackage{marvosym}
\usepackage{url}
\usepackage{amsbsy}
\usepackage{breakurl}

\usepackage{fullpage}
\numberwithin{equation}{section}

\DeclareFontFamily{U}{matha}{\hyphenchar\font45}
\DeclareFontShape{U}{matha}{m}{n}{
      <5> <6> <7> <8> <9> <10> gen * matha
      <10.95> matha10 <12> <14.4> <17.28> <20.74> <24.88> matha12
      }{}
\DeclareSymbolFont{matha}{U}{matha}{m}{n}
\DeclareFontSubstitution{U}{matha}{m}{n}
\DeclareFontFamily{U}{mathx}{\hyphenchar\font45}
\DeclareFontShape{U}{mathx}{m}{n}{
      <5> <6> <7> <8> <9> <10>
      <10.95> <12> <14.4> <17.28> <20.74> <24.88>
      mathx10
      }{}
\DeclareSymbolFont{mathx}{U}{mathx}{m}{n}
\DeclareFontSubstitution{U}{mathx}{m}{n}
\DeclareMathDelimiter{\vvvert}{0}{matha}{"7E}{mathx}{"17}

\newcommand{\bv}{{\mathbf v}}

\newcommand{\floor}[1]{\lfloor #1 \rfloor}
\newcommand{\roof}[1]{\lceil #1 \rceil}

 \newcommand{\tY}{\tilde{Y}}
\newcommand{\MP}{\mathcal{P}}
\newcommand{\MA}{\mathcal{A}}
\newcommand{\MB}{\mathcal{B}}
\newcommand{\ME}{\mathcal{E}}
\newcommand{\mfB}{\mathfrak{B}}
\newcommand{\MF}{\mathcal{F}}
\newcommand{\MD}{\mathcal{D}}
\newcommand{\ML}{\mathcal{L}}
\newcommand{\MV}{\mathcal{V}}
\newcommand{\tBW}{\tilde{\mathbf{W}}}

\newcommand{\BBW}{\mathbb{W}}
\newcommand{\BBN}{\mathbb{N}}
\newcommand{\BBX}{\mathbb{X}}

\newcommand{\ssep}{\mid}

\newcommand{\mrd}{\mathop{}\!\mathrm{d}}

\newcommand{\LLambda}{{\Omega}}
\newcommand{\llambda}{{\omega}}

\newtheorem{thm}{Theorem}[section]
\newtheorem{assumption}[thm]{Assumption}
\newtheorem{lem}[thm]{Lemma}
\newtheorem{cor}[thm]{Corollary}
\newtheorem{rem}[thm]{Remark}
\newtheorem{prop}[thm]{Proposition}

\newtheorem{defi}[thm]{Definition}

\newcommand{\E}{\mathbb{E}}
\newcommand{\R}{\mathbb{R}}
\newcommand{\C}{\mathbb{C}}
\newcommand{\X}{\mathbb{X}}
\newcommand{\BX}{\mathbf{X}}
\newcommand{\BY}{\mathbf{Y}}

\newcommand{\op}{\mathcal{P}}

\newcommand{\I}{\mathcal{I}}

\newcommand{\BL}{\mathbf{L}}
\newcommand{\tBL}{\widetilde{\mathbf{L}}}
\newcommand{\BW}{\mathbf{W}}
\newcommand{\eps}{\varepsilon}
 \def\var{\text{-}\textrm{var}}
 \def\Hol{\text{-}\textrm{H\"ol}}
 \def\var#1{#1\textnormal{-var}}
 \def\Hol#1{#1\textnormal{-H{\"o}l}}

 \newcommand{\Leb}{\operatorname{Leb}}

\begin{document}

\title{Deterministic homogenization under optimal moment assumptions for fast-slow systems.  Part 2.}

\author
{I. Chevyrev\thanks{School of Mathematics,
University of Edinburgh,
Edinburgh, EH9 3FD,
UK.}
\and
P.~K. Friz\thanks{Institut f\"ur Mathematik, Technische Universit\"at Berlin, and Weierstra\ss --Institut f\"ur Angewandte Analysis und Stochastik, Berlin, Germany}
\and
A. Korepanov\thanks{Mathematics Department, University of Exeter, Exeter, EX4 4QF, UK}
\and
I. Melbourne\thanks{Mathematics Institute, University of Warwick, Coventry, CV4 7AL, UK}
\and
H. Zhang\thanks{Institute of Mathematics, Fudan University, Shanghai, 200433, China}
}

\date{10 June 2020. Updated 24 June 2021.}

\maketitle

\begin{abstract}
We consider deterministic homogenization for discrete-time fast-slow systems of the form
$$ X_{k+1} = X_k + n^{-1}a_n(X_k,Y_k) + n^{-1/2}b_n(X_k,Y_k)\;, \quad Y_{k+1} = T_nY_k\;$$
and give conditions under which the dynamics of the slow equations converge weakly to an It\^o diffusion $X$ as $n\to\infty$.
The drift and diffusion coefficients of the limiting stochastic differential equation satisfied by $X$ are given explicitly.
This extends the results of [Kelly--Melbourne, J. Funct.\ Anal.\ 272 (2017) 4063--4102] from the continuous-time case to the discrete-time case.
Moreover, our methods ($p$-variation rough paths) work under optimal moment assumptions.

Combined with parallel developments on martingale approximations for families of nonuniformly expanding maps in Part 1 by Korepanov, Kosloff \& Melbourne, we obtain optimal homogenization results when $T_n$ is such a family of maps.
\end{abstract}

%\tableofcontents

%\setcounter{section}{-1}
%%%%%%%%%%%%%%%%%%%%%%%%%%%%%%%%%%%%%%%%%%%%%%%%%%%%%%%%%%%%%%%%%%%%%%%%%%%%%%%%%%%%%%%%%%%%%%%%%%%%%%%%%%%%%%%%%%%%%%%%

\section{Introduction}

In this article, we are primarily concerned with homogenization of deterministic, discrete-time, fast-slow systems of the form
\begin{equation}\label{eq:FS_intro}
X_{k+1}^{(n)} = X_k^{(n)} + n^{-1}a_n(X_k^{(n)},Y_k^{(n)}) + n^{-1/2}b_n(X_k^{(n)},Y_k^{(n)})\;, \quad Y_{k+1}^{(n)} = T_nY_k^{(n)}\;,
\end{equation}
where $X_k^{(n)}$ takes values in $\R^d$, $Y_k^{(n)}$ takes values in a metric space $\Lambda$,
and $a_n,b_n : \R^d\times \Lambda \to \R^d$ and $T_n : \Lambda\to \Lambda$ are suitable functions.
The only source of randomness in the dynamics is the initial condition $Y_0^{(n)}$ which we sample from a (not necessarily ergodic) probability measure $\lambda_n$ on $\Lambda$.

Our main result, Theorem~\ref{thm:discrete_FS_conv}, provides sufficient conditions for the dynamics $x_n(t) = X^{(n)}_{\floor{nt}}$ to converge in law (which we write in symbols as $x_n \to_{\lambda_n} X$), with respect to the uniform topology, to the solution of a stochastic differential equation (SDE)
\begin{equation} \label{eq:SDE}
\mrd X = \tilde a(X) \mrd t + \sigma(X) \mrd B
\end{equation}
with explicit formulae for the coefficients $\tilde a$ and $\sigma$.
Our assumptions on the system involve only moment bounds and a suitable (iterated) weak invariance principle on the fast dynamics $T_n$.
In the companion paper, \emph{Part 1}~\cite{KKMprep}, it is shown how these assumptions can be verified for a large class of families $T_n$ of nonuniformly hyperbolic dynamical systems.
See Section~\ref{subsec:example} for an illustrative example of a system to which our results apply.

The programme to study homogenization of deterministic systems of the form~\eqref{eq:FS_intro} was initiated in~\cite{MS11}, and has seen recent growth in a number of works, including~\cite{GM13b,KM16,KM17}.
See our survey paper~\cite{CFKMZ19} for an overview.
The contribution of this article is three-fold.
The first two of these contributions are novel even when we suppose that $a_n\equiv a$, $b_n\equiv b$, $T_n\equiv T$ are independent of $n$.
First, we are able to deal with discrete-time dynamics in the same way as continuous-time dynamics.
This should be compared to~\cite{KM16, Kelly16} in which results for discrete-time dynamics are only obtained in the special case $a(x,y) = a(x)$, $b(x,y)=b(x)v(y)$ and the case of general $a,b$ is only handled for continuous-time dynamics in~\cite{BC17, KM17}.

Second, we are able to work under optimal moment assumptions, and our results apply to the full range of systems in which one expects a weak invariance principle to hold for the fast dynamics.
This extends (even for continuous-time dynamics) the results of~\cite{KM16,KM17,BC17} in which only a subrange can be handled
(specifically, in Assumptions~\ref{assump:discrete_W_assump_n_indep}\ref{point:Kolm_assump_n_indep} and~\ref{assump:discrete_W_assump}\ref{point:Kolm_assump} 
it now suffices that $q>1$ rather than $q>3$ as was the case previously;
in particular the required
control on ordinary moments is reduced from $6+\eps$ to $2+\eps$).
In~\cite{CFKMZ19} we indicate a simplified version of this second contribution for the case $a(x,y) = a(x)$, $b(x,y)=b(x)v(y)$.

In particular, when $T_n=T$ is independent of $n$, our results apply to uniformly hyperbolic (Axiom~A) systems~\cite{Smale67}, and to large classes of nonuniformly hyperbolic systems~\cite{Young98,Young99}.
A detailed account of discrete-time dynamical systems $T$ for which our assumptions are verified can be found in~\cite[Sec.~10]{KM16} and~\cite[Sec.~1]{KM17}; our results on homogenization apply to all the systems therein without restriction on the form of $a$ and $b$ and under optimal moment bounds.

Our third contribution is to incorporate families of fast dynamical systems $T_n$ and measures $\lambda_n$.
Such fast-slow systems were studied in the situation of exact multiplicative noise (which does not require rough path theory) in~\cite{KKM18}.
As mentioned above, in Part~1~\cite{KKMprep}, the assumptions in the current paper are verified for a range of families $T_n$.

The main tool in showing convergence of the system~\eqref{eq:FS_intro} is rough path theory~\cite{Lyons98}, which we apply in the c{\`a}dl{\`a}g setting in conjunction with the method in~\cite{KM17}.
We note here that our second contribution outlined above (optimal moment assumptions) is due to switching from $\alpha$-H{\"o}lder to $p$-variation rough path topologies (which is analogous to the mode of convergence in the classical Donsker theorem, see e.g.~\cite[Sec.~3.2]{CFKMZ19}).
Our results employ the stability of ``forward'' (It{\^o}) rough differential equations (RDEs) with jumps recently studied in~\cite{FZ18}, which we extend herein to the Banach space setting (though we restrict attention to the case of level-2 rough paths).
The works~\cite{FS17, Chevyrev18, CF19} also study RDEs in the presence of jumps, but primarily focus on ``geometric'' (Marcus) notions of solution.

\subsection{Illustrative example}\label{subsec:example}

Let $\Lambda = [0,1]$. For $\gamma \geq 0$, we consider the intermittent map $T \colon \Lambda \to \Lambda$,
\begin{equation} \label{eq:LSV}
    T y = \begin{cases}
        y (1 + 2^\gamma y^\gamma)\;, \quad & y \leq 1/2\;, \\
        2 y - 1\;,                 & y > 1/2\;.
    \end{cases}
\end{equation}
This is a prototypical example of a slowly mixing dynamical system~\cite{PomeauManneville80}; the specific example is due to~\cite{LSV99}.
We describe in this subsection the homogenization results for the associated fast-slow systems which follow from this paper together with Part~1~\cite{KKMprep},
and compare these results with earlier works.

For $\gamma < 1$, there exists a unique $T$-invariant ergodic absolutely
continuous probability measure $\mu$.
We further restrict to  $\gamma < 1/2$, where the central limit theorem holds:
for $v \colon \Lambda \to \R^m$ H\"older continuous with $\int_\Lambda v\mrd\mu=0$,
the random variables $n^{-1/2} \sum_{j=0}^{n-1} v \circ T^j$
defined on the probability space $(\Lambda, \mu)$
converge in law to a (typically nondegenerate) normal distribution.

Consider a discrete-time fast-slow system of the form~\eqref{eq:FS_intro}
with $a_n\equiv a$, $b_n\equiv b$, $T_n\equiv T$ independent of $n$, and $T$ such an intermittent map.
Here $a, b \colon \R^d \times \Lambda \to \R^d$ are suitably regular
functions such that $\int b(x, y) \mrd \mu(y) = 0$ for all $x\in\R^d$.
Define the c{\`a}dl{\`a}g random process $x_n(t) = X^{(n)}_{[nt]}$.

Prior results establish convergence $x_n\to_\mu X$, for $X$ the solution of an SDE~\eqref{eq:SDE}, provided that $b$ is a product $b(x,y) = h(x) v(y)$ with
$h \colon \R^d \to \R^{d \times m}$ sufficiently smooth and $v \colon \Lambda \to \R^m$ as above.
It was proved first for $\gamma < \frac{2}{11}$ in~\cite{KM16} using a discrete-time version of H\"older
rough paths~\cite{Kelly16}, %then improved to $\gamma < \frac14$ by obtaining optimal moment control in~\cite{KKMprep}, 
and extended to the range $\gamma < \frac25$ in~\cite{CFKMZ19} using $p$-variation
rough paths with jumps~\cite{FZ18}. 
%See~\cite{CFKMZ19} for further history and discussion.

Part~1 by Korepanov~{\em et al.}~\cite{KKMprep} develops the smooth ergodic theory side of things and together with~\cite{CFKMZ19} covers the optimal range $\gamma\in(0,\frac12)$ in the case when $b(x,y)=h(x)v(y)$ is a product.
In the current paper, Theorem~\ref{thm:discrete_FS_conv_n_indep} enables two improvements to these prior results.  First,
the restriction that $b$ is a product is now redundant.
Second, we show that
$x_n\to_\lambda X$ for an enlarged class of measures $\lambda$.
  In particular, the hypotheses of Theorem~\ref{thm:discrete_FS_conv_n_indep} are verified in~\cite{KKMprep}
for the most natural choice $\lambda=\Leb$ for all $\gamma\in(0,\frac12)$.  

%The forthcoming paper~\cite{KKMprep} 
In addition, we consider the general setting~\eqref{eq:FS_intro} where $T$, $a$, $b$ and $\lambda$ are allowed to depend on $n\in\BBN\cup\{\infty\}$.  This requires our main result Theorem~\ref{thm:discrete_FS_conv}.
For example, consider the case where $T_n$ is a family of intermittent maps with parameters $\gamma_n$ limiting on $\gamma_\infty\in(0,1/2)$.  
In~\cite{KKM18}, convergence results of the form $x_n\to_{\mu_n}X$ and $x_n\to_{\mu_\infty}X$ were obtained for the special case $b_n(x,y)=h_n(x)v_n(y)$ with $h_n$ exact.
Theorem~\ref{thm:discrete_FS_conv} combined with results in~\cite{KKMprep}
yields the same convergence results without restrictions on $b_n$, and also shows that $x_n\to_{\Leb} X$.
The coefficients $\tilde a$ and $\sigma$ in~\eqref{eq:SDE} are given by
\begin{align} \label{eq:coeffs1} 
\tilde a(x) & =
\int_\Lambda a_\infty(x,y)\mrd \mu_\infty(y)+\sum_{k=1}^d\sum_{\ell=1}^\infty\int_\Lambda  b_\infty^k(x,y)\frac{\partial b_\infty}{\partial x_k}(x,T_\infty^\ell y)\mrd \mu_\infty(y)\;, \\
\begin{split}\label{eq:coeffs2}
\sigma(x)^2 & =
\int_\Lambda b_\infty(x,y)\otimes b_\infty(x,y)\mrd \mu_\infty(y)
\\ & \qquad +
\sum_{\ell=1}^\infty \int_\Lambda \big\{b_\infty(x,y)\otimes b_\infty(x,T_\infty^\ell y)+
b_\infty(x,T_\infty^\ell y)\otimes b_\infty(x,y)\big\}\mrd \mu_\infty(y)\;,
\end{split}
\end{align}
where $b_\infty^k$ is the $k$th column of $b_\infty$.
The details of how to apply Theorem~\ref{thm:discrete_FS_conv} and the results in Part~1~\cite{KKMprep} are given in Section~\ref{subsec:LSV}.

\bigskip
The article is structured as follows.
In Section~\ref{sec:discrete_fs} we state the main result of this article, Theorem~\ref{thm:discrete_FS_conv}, which gives precise conditions for the dynamics~\eqref{eq:FS_intro} to converge to the solution of an SDE.
In Section~\ref{sec:RP} we collect the necessary material on c{\`a}dl{\`a}g rough path theory in the Banach space setting.
In Section~\ref{sec:discrete_proof} we prove Theorem~\ref{thm:discrete_FS_conv}.
In Section~\ref{sec:cont_fs} we give the version of Theorem~\ref{thm:discrete_FS_conv} for the continuous-time dynamics.
In Appendix~\ref{appendix:Besov}, we give a Banach-space version of homogeneous Besov-variation and Besov--H{\"o}lder rough path embeddings.

\bigskip
\noindent
\textbf{Acknowledgements.}
I.C. is funded by a Junior Research Fellowship of St John's College, Oxford.
P.K.F. acknowledges partial support from the ERC, CoG-683164, the Einstein Foundation Berlin, and DFG research unit FOR2402.
A.K. and I.M. acknowledge partial support from the European Advanced Grant StochExtHomog (ERC AdG 320977).
A.K. is also supported by an Engineering and Physical Sciences Research Council grant
EP/P034489/1.
H.Z. is supported by the Chinese National Postdoctoral Program for Innovative Talents No: BX20180075.
I.C., A.K., and H.Z. thank the Institute f{\"u}r Mathematik, TU Berlin, for its hospitality.

\section{Discrete-time fast-slow systems. Statement of the main result}
\label{sec:discrete_fs}

In this section we state our main result, Theorem~\ref{thm:discrete_FS_conv}.
In fact, we first state a simplified version, Theorem~\ref{thm:discrete_FS_conv_n_indep}, which applies to the case that $a_n$, $b_n$, $T_n$ and $\lambda_n$ do not depend on $n$.
We state the results separately not only because it eases our presentation, but also because Theorem~\ref{thm:discrete_FS_conv_n_indep} is slightly stronger than the naive restriction of Theorem~\ref{thm:discrete_FS_conv} to the $n$-independent case (namely Assumption~\ref{assump:discrete_W_assump_n_indep} below is weaker than the naive restriction of Assumption~\ref{assump:discrete_W_assump}).
In Subsection~\ref{subsec:LSV}, we show that our assumptions are satisfied for the intermittent maps considered in Section~\ref{subsec:example}.

For the remainder of this section, we fix a metric space $(\Lambda,\rho)$.

\begin{defi}\label{def:discrete_setup}
For $\kappa \in [0,1)$ and $m \geq 1$, let $C^\kappa(\Lambda,\R^m)$ denote the space of continuous $\R^m$-valued functions on $\Lambda$ such that
\[
| v |_{C^\kappa} := \sup_{y \in \Lambda} |v(y)| + \sup_{y,y' \in \Lambda}\frac{|v(y)-v(y')|}{\rho(y,y')^\kappa} < \infty\;.
\]
We write $C^\kappa(\Lambda)$ whenever $m=1$.
For $\alpha \geq 0$, define $C^{\alpha,\kappa}(\R^d\times \Lambda,\R^d)$
to be the space of functions $a=a(x,y):\R^d\times \Lambda\to\R^d$ such that 
\[
|a|_{C^{\alpha,\kappa}} := \sum_{|k|\leq \floor\alpha} \sup_{x \in \R^d} |D^k a(x,\cdot)|_{C^\kappa} + \sum_{|k|=\floor\alpha} \sup_{x,x'\in\R^d} \frac{|D^k a(x,\cdot) - D^k a(x',\cdot)|_{C^\kappa}}{|x-x'|^{\alpha-\floor\alpha}} < \infty\;,
\]
where $D^k$ acts on the $x$ component.
\end{defi}

For the remainder of the section, we fix parameters $q \in (1,\infty]$, $\kappa, \bar\kappa \in (0,1)$, and $\alpha > 2+\frac{d}{q}$.
For $T>0$, a metric space $E$, and a c{\`a}dl{\`a}g function $f\colon[0,1]\to E$, we define $f^-\colon[0,1]\to E$ by $f^-(t)=\lim_{s\uparrow t}f(s)$ for $t\in(0,T]$ and $f^-(0)=f(0)$.

\subsection{\texorpdfstring{$n$}{n}-independent case}
\label{subsec:n_indep}

We now describe the assumptions and preliminary results required to state Theorem~\ref{thm:discrete_FS_conv_n_indep}.
We fix $a \in C^{1+\bar\kappa,0}(\R^d\times \Lambda,\R^d)$ and $b\in C^{\alpha,\kappa}(\R^d\times \Lambda, \R^d)$, and consider for every integer $n \geq 1$ the discrete-time dynamical system posed on $\R^d\times \Lambda$
\begin{equation} \label{eq:discreteFS_n_indep}
X^{(n)}_{k+1} = X^{(n)}_k + n^{-1}a(X_k^{(n)},Y_k) + n^{-1/2}b(X_k^{(n)},Y_k)\;,\quad
Y_{k+1} = T Y_k\;,
\end{equation}
where $T:\Lambda\to \Lambda$ is a Borel measurable map, $X^{(n)}_0 = \xi_n \in \R^d$, and $Y_0$ is drawn randomly from a Borel probability measure $\lambda$ on $\Lambda$.
Our first assumption deals with the function $a$.

\begin{assumption}\label{assump:conv_drift_n_indep}
There exists $\bar a \in C^{1+\bar\kappa}(\R^d,\R^d$) such that for all $x\in\R^d$
\[
\Big|n^{-1}\sum_{k=0}^{n-1} a(x,Y_k) - \bar a(x)\Big|
\to_{\lambda}0 \quad \textnormal{ as } \quad n\to\infty\;.
\]
%where $V_n(t) = n^{-1}\sum_{k=0}^{\floor{tn}-1} a(\cdot, Y_{k})$.
\end{assumption}

To state our assumption on $b$, we need to introduce further notation.
For $v,w\in C^\kappa(\Lambda,\R^m)$ and $0\le s\le t\le1$,
 define 
$W_{v,n}(t) \in \R^m$ 
and $\BBW_{v,w,n}(s,t)\in \R^{m \times m}$ by
\begin{equation}\label{eq:discrete_BBW_def_n_indep}
W_{v,n}(t) = n^{-1/2}\!\!\!\!\sum_{0\le k<\floor{nt}}\!\!\!\! v(Y_k) \;, \quad
\BBW_{v,w,n}(s,t) = \int_s^t  (W^-_{v,n} (r) - W_{v,n} (s)) \otimes \mrd W_{w,n}(r) \;,
\end{equation}
where we recall that $W^-_{v,n} (r)=\lim_{s\uparrow r} W_{v,n}(s)$.
Note in particular that
\begin{equation}\label{eq:BBW_at_t_n_indep}
\BBW_{v,w,n}(t) := \BBW_{v,w,n}(0,t) = n^{-1}\!\!\!\!\sum_{0\le k<\ell<\floor{nt}}\!\!\!\! v(Y_k)\otimes w(Y_\ell)
\;.
\end{equation}
Whenever $v=w$, we write simply $\BBW_{v,n}$ for $\BBW_{v,v,n}$.

For a subspace $C^\kappa_0(\Lambda)$ of $C^\kappa(\Lambda)$, we let $C^\kappa_0(\Lambda,\R^m)$ denote the space of all $v\in C^\kappa(\Lambda,\R^m)$ such that $v^i \in C^\kappa_0(\Lambda)$ for all $i=1,\ldots, m$, and we let $C^{\alpha,\kappa}_0(\R^d\times \Lambda,\R^d)$ denote the subspace of all $f\in C^{\alpha,\kappa}(\R^d\times \Lambda,\R^d)$ for which $f(x,\cdot) \in C^\kappa_0(\Lambda,\R^d)$ for all $x\in\R^d$.

\begin{assumption}\label{assump:discrete_W_assump_n_indep}
There exists a closed subspace $C_0^\kappa(\Lambda)$ of $C^\kappa(\Lambda)$ such that $b \in C_0^{\alpha,\kappa}(\R^d\times \Lambda,\R^d)$ and such that
\begin{enumerate}[label=(\roman*)]
\item\label{point:Kolm_assump_n_indep} for all $v,\,w\in C_0^\kappa(\Lambda)$ there exists
$K=K_{v,w,q}>0$ such that for all $n \geq 1$ and $0\leq k,\ell \leq n$
\[
|W_{v,n}(k/n) - W_{v,n}(\ell/n)|_{L^{2q}(\lambda)} \le Kn^{-1/2}|k-\ell|^{1/2}
\]
and
\[
|\BBW_{v,w,n}(k/n,\ell/n)|_{L^{q}(\lambda)} \le Kn^{-1}|k-\ell|\;.
\]

\item\label{point:fdd_assump_n_indep} there exists a bilinear operator $\mfB_0:C_0^\kappa(\Lambda)\times C_0^\kappa(\Lambda) \to\R$ such that for every $m\ge1$ and every $v\in C^\kappa_0(\Lambda,\R^m)$,
it holds that $(W_{v,n},\BBW_{v,n})\to_{\lambda} (W_v,\BBW_v)$ as $n\to\infty$
in the sense of finite-dimensional distributions, where $W_v$ is an $\R^m$-valued Brownian motion and
\[
\BBW^{ij}_v(t)=\int_0^t W_v^i\mrd W_v^j+\mfB_0(v^i,v^j)t\;.
\]
\end{enumerate}  
\end{assumption}

\begin{rem}\label{rem:optimal_moments_n_indep}
One should compare Assumption~\ref{assump:discrete_W_assump_n_indep}\ref{point:Kolm_assump_n_indep} to~\cite[Thm.~9.1]{KM16} and~\cite[Assump.~2.2]{KM17} in which one imposes the restriction $q > 3$.
As mentioned in the introduction, we are able to deal with the optimal moment condition $q>1$ by working with $p$-variation rather than H{\"o}lder rough path topologies.
\end{rem}

\begin{rem}\label{rem:ergodic_measures}
Assumptions~\ref{assump:conv_drift_n_indep} and~\ref{assump:discrete_W_assump_n_indep} are verified for a large class of dynamical systems in~\cite[Sec.~10]{KM16} and~\cite[Sec.~1]{KM17}.
In these references, as in Subsection~\ref{subsec:example}, there is a $T$-invariant ergodic Borel probability measure $\mu$ on $\Lambda$, and we choose $C_0^\kappa(\Lambda)=\{v\in C^\kappa(\Lambda) :\int_\Lambda v \mrd \mu = 0\}$
and $\bar a=\int_\Lambda a(\cdot,y)\mrd\mu(y)$.

The measure $\mu$ plays no role in the proof of Theorem~\ref{thm:discrete_FS_conv_n_indep} and
hence we do not mention it in our assumptions.  
\end{rem}

\begin{rem}\label{rem:stationary_simplification_n_indep}
Under the assumption that $\lambda$ is $T$-stationary, the simpler bounds
\[
|W_{v,n}(1)|_{L^{2q}(\lambda)} \le K 
\quad\text{and} \quad
|\BBW_{v,w,n}(1)|_{L^{q}(\lambda)} \le K 
\quad\text{for all $n \geq 1$}
\]
imply Assumption~\ref{assump:discrete_W_assump_n_indep}\ref{point:Kolm_assump_n_indep}.
\end{rem}

\begin{prop} \label{prop:moments_discrete_n_indep}
Suppose Assumption~\ref{assump:discrete_W_assump_n_indep}\ref{point:Kolm_assump_n_indep} holds.
Then there exists $K>0$ such that for all $n \geq 1$, $0 \leq k,\ell\leq n$, and $v,\,w\in C_0^\kappa(\Lambda)$,
\begin{align}
\Big|W_{v,n}(k/n)-W_{v,n}(\ell/n)\Big|_{L^{2q}(\lambda)} &\le K |v|_{C^\kappa} n^{-1/2}|k-\ell|^{1/2}\;,
\nonumber
\\
\Big|\BBW_{v,w,n}(k/n,\ell/n)\Big|_{L^{q}(\lambda)} &\le K |v|_{C^\kappa}|w|_{C^\kappa} n^{-1}|k-\ell|\;. 
\label{eq:W_v_bound_n_indep}
\end{align}
\end{prop}

\begin{proof}
As in~\cite[Prop.~2.7]{KM17}, the constants in Assumption~\ref{assump:discrete_W_assump_n_indep}\ref{point:Kolm_assump_n_indep} have the required dependence on 
$|v|_{C^\kappa}$ and $|w|_{C^\kappa}$ by the uniform boundedness principle.
\end{proof}

\begin{prop} \label{prop:auto_n_indep}
Suppose Assumption~\ref{assump:discrete_W_assump_n_indep} holds.  Then 
\begin{enumerate}[label=(\alph*)]
\item\label{point:mfB_prime_def_n_indep} for all $v\in C^\kappa_0(\Lambda,\R^m)$, the limit $\lim_{n\to\infty}n^{-1}\sum_{k=0}^{n-1}\E_\lambda (v^iv^j)(Y_k)$ exists and the covariance of $W_v$ is given by
\[
\E W_v^i(1)W_v^j(1)=\mfB(v^i,v^j)+\mfB(v^j,v^i)\;,
\]
where
\[
\mfB(v^i,v^j)=\mfB_0(v^i,v^j)+\frac12 \lim_{n\to\infty}n^{-1}\sum_{k=0}^{n-1}\E_\lambda (v^iv^j)(Y_k)\;,
\]
\item\label{point:mfB_bound_n_indep} the bilinear operators
$\mfB,\,\mfB_0:
C^\kappa_0(\Lambda) \times C^\kappa_0(\Lambda)\to\R$ are bounded.
\end{enumerate}
\end{prop}

\begin{proof}
\ref{point:mfB_prime_def_n_indep}
It follows from Assumption~\ref{assump:discrete_W_assump_n_indep}
that
\[
\E_{\lambda} W_{v,n}^i(1) W_{v,n}^j(1)\to \E W_{v}^i(1) W_{v}^j(1)\;,
\]
and
\begin{align} \label{eq:pointa}
\E_{\lambda}\BBW_{v,n}^{ij}(1)\to \E \BBW_{v}^{ij}(1)= \mfB_0(v^i,v^j)\;,
\end{align}
where we have used the fact that It\^o integrals have zero mean.
By~\eqref{eq:BBW_at_t_n_indep}, we have
\[
W_{v,n}^i(1) W_{v,n}^j(1)= \BBW_{v,n}^{ij}(1)+\BBW_{v,n}^{ji}(1)+
n^{-1}\sum_{k=0}^{n-1}(v^iv^j)(Y_k)\;.
\]
Taking expectations on both sides
and letting $n\to\infty$ yields the desired result.

\noindent
\ref{point:mfB_bound_n_indep}
Boundedness of $\mfB_0$ follows from~\eqref{eq:pointa} and~\eqref{eq:W_v_bound_n_indep} with $k=0$, $\ell=n$.
By definition of $\mfB$, we have
$|\mfB(v,w)| \leq |\mfB_0(v,w)| + \frac12|v|_{C^0}|w|_{C^0}$, yielding boundedness of $\mfB$.
\end{proof}

\begin{lem}\label{lem:Lip_sqrt_n_indep}
Suppose Assumption~\ref{assump:discrete_W_assump_n_indep} holds.
Then the quadratic form
\begin{align*}
\Sigma^{ij}(x) = \mfB(b^i(x,\cdot),b^j(x,\cdot)) + \mfB(b^j(x,\cdot),b^i(x,\cdot))\;, \quad i,j =1, \dots, d\;,
\end{align*}
is positive semi-definite and the unique positive semi-definite $\sigma$ satisfying $\sigma^2 = \Sigma$ is Lipschitz.
\end{lem}

\begin{proof}
Positive semi-definiteness of $\Sigma$ follows from Proposition~\ref{prop:auto_n_indep}\ref{point:mfB_prime_def_n_indep}.
Moreover, $b$ lies in $C^{\alpha,\kappa}(\R^d\times \Lambda,\R^d)$ with $\alpha>2+\frac{d}{q}\ge 2$, so
$\Sigma$ is $C^2$ with globally bounded derivatives to second order.
The conclusion now follows from~\cite[Thm.~5.2.3]{StroockVaradhan06}.
\end{proof}

As a consequence of Lemma~\ref{lem:Lip_sqrt_n_indep} and~\cite[Cor.~5.1.2]{StroockVaradhan06}, for a Brownian motion $B$ on $\R^d$ and a Lipschitz function $\tilde a : \R^d \to \R^d$, there is a unique strong solution to the SDE
\begin{equation}\label{eq:discrete nonprod sde_n_indep}
\mrd X = \tilde a(X)\mrd t + \sigma(X)\mrd B\;, \quad X(0) = \xi\;.
\end{equation}
In particular, the SDE~\eqref{eq:discrete nonprod sde} has uniqueness in law.

\begin{thm}\label{thm:discrete_FS_conv_n_indep}
Suppose that Assumptions~\ref{assump:conv_drift_n_indep} and~\ref{assump:discrete_W_assump_n_indep} hold and that $\lim_{n \to \infty}\xi_n = \xi \in \R^d$.
Define the c{\`a}dl{\`a}g path
\begin{equation}\label{eq:x_n discrete_n_indep}
x_{n} : [0,1] \to \R^d\;, \quad x_{n}(t) = X^{(n)}_{\floor{nt}}\;.
\end{equation}
Then $x_n \to_{\lambda} X$ in the uniform topology as $n \to \infty$, where $X$ is a weak solution of the SDE~\eqref{eq:discrete nonprod sde_n_indep},
where $B$ is a standard Brownian motion in $\R^d$, $\sigma$ is defined as in Lemma~\ref{lem:Lip_sqrt_n_indep}, and $\tilde a$ is the Lipschitz function given by
\begin{align*}
\tilde a^i(x) = \bar a^i(x) +  \sum_{k=1}^d \mfB_0(b^k(x,\cdot),\partial_k b^i(x,\cdot))\;,
\quad i=1,\dots,d\;.
\end{align*}
\end{thm}

We omit the proof of Theorem~\ref{thm:discrete_FS_conv_n_indep}, which follows from trivial modifications to the proof in Section~\ref{sec:discrete_proof}
 of Theorem~\ref{thm:discrete_FS_conv}.

\subsection{General case}\label{subsec:general_case}

We now state the assumptions and preliminary results required for our main result, Theorem~\ref{thm:discrete_FS_conv}.
We fix functions $a_n \in C^{1+\bar\kappa,0}(\R^d\times \Lambda,\R^d)$, and $b_\infty, b_n \in C^{\alpha,\kappa}(\R^d\times \Lambda,\R^d)$ satisfying
\[
\sup_{n \geq 1} |a_n|_{C^{1+\bar\kappa,0}} + |b_n|_{C^{\alpha,\kappa}} < \infty\;, \quad \lim_{n \to \infty} |b_n - b_\infty|_{C^{\alpha,\kappa}} = 0\;.
\]
For $n \geq 1$, we are interested in the discrete-time fast-slow system~\eqref{eq:FS_intro} 
where $T_n:\Lambda\to \Lambda$ is a measurable map, $X^{(n)}_0 = \xi_n \in \R^d$, and $Y^{(n)}_0$ is drawn randomly from a Borel probability measure $\lambda_n$ on $\Lambda$.

\begin{assumption}\label{assump:conv_drift}
There exists $\bar a \in C^{1+\bar\kappa}(\R^d,\R^d$) such that, for all $t\in[0,1]$ and $x\in\R^d$,
\[
|V_n(t)(x)-t\bar a(x)|\to_{\lambda_n} 0 \quad \textnormal{ as } \quad n\to\infty\;,
\]
where $V_n(t) = n^{-1}\sum_{k=0}^{\floor{tn}-1} a_n(\cdot, Y^{(n)}_k)$.
\end{assumption}

As in \eqref{eq:discrete_BBW_def_n_indep}, for $v,w\in C^\kappa(\Lambda,\R^m)$ and $0\le s\le t\le 1$, define 
$W_{v,n}(t) \in \R^m$, 
and $\BBW_{v,w,n}(s,t) \in \R^{m \times m}$ by
\begin{equation}\label{eq:discrete_BBW_def_n_dep}
    \begin{aligned}
        W_{v,n}(t) & = n^{-1/2}\!\!\!\!\sum_{0\le k<\floor{nt}}\!\!\!\! v(Y_k^{(n)})
        \;, \\
        \BBW_{v,w,n}(s,t) &= \int_s^t  (W^-_{v,n} (r) - W_{v,n} (s)) \otimes \mrd W_{w,n}(r) \;,
    \end{aligned}
\end{equation}
where we recall that $W^-_{v,n} (r)=\lim_{s\uparrow r} W_{v,n}(s)$.
Whenever $v=w$, we again write $\BBW_{v,n}$ for $\BBW_{v,v,n}$.

Recall our notational convention about subspaces $C^\kappa_0(\Lambda)$ of $C^\kappa(\Lambda)$ introduced before Assumption~\ref{assump:discrete_W_assump_n_indep}.
% Given a family of subspaces $(C^\kappa_n(\Lambda))_{n \in \BBN\cup\{\infty\}}$ of $C^\kappa(\Lambda)$, we call $\bv=(v_n)_{n \in \BBN\cup\{\infty\}}$ a $C^\kappa_n(\Lambda,\R^m)$-family if $v_n \in C^\kappa_n(\Lambda,\R^m)$ and 
% $\lim_{n \to \infty} |v_n-v_\infty|_{C^\kappa} = 0$.

\begin{assumption}\label{assump:discrete_W_assump}
There exists a closed subspace $C_n^\kappa(\Lambda)$ of $C^\kappa(\Lambda)$ for each $n \in \BBN\cup\{\infty\}$ such that $b_n \in C_n^{\alpha,\kappa}(\R^d\times \Lambda,\R^d)$, and
\begin{enumerate}[label=(\roman*)]
\item\label{point:Kolm_assump} for all $v = (v_1,\ldots),\,w=(w_1,\ldots) \in \prod_{n \in \BBN} C_n^\kappa(\Lambda)$ with
\[
\sup_n |v_n|_{C^\kappa}+|w_n|_{C^\kappa} < \infty\;,
\]
there exists
$K = K_{v,w,q}>0$ such that for all $n\in \BBN$ and $0 \leq k,\ell \leq n$
\[
|W_{v_n,n}(k/n)-W_{v_n,n}(\ell/n)|_{L^{2q}(\lambda_n)} \le K n^{-1/2}|k-\ell|^{1/2}
\]
and
\[
|\BBW_{v_n,w_n,n}(k/n,\ell/n)|_{L^{q}(\lambda_n)} \le K n^{-1} |k-\ell|\;.
\]

\item\label{point:fdd_assump} there exist bounded bilinear operators $\mfB_1,\mfB_2:C_\infty^\kappa(\Lambda)\times C_\infty^\kappa(\Lambda) \to\R$ such that for every $m\ge1$ and all 
$\bv=(v_n)_{n \in \BBN\cup\{\infty\}}$ with
$v_n\in C^\kappa_n(\Lambda,\R^m)$ and 
$\lim_{n \to \infty} |v_n-v_\infty|_{C^\kappa} = 0$,
%$C^\kappa_n(\Lambda,\R^m)$-family $\bv=(v_n)_{n \in \BBN\cup\{\infty\}}$,
\begin{enumerate}[label=(\alph*)]
\item\label{subpoint:B_1}
$\lim_{n\to\infty}n^{-1}\sum_{k=0}^{n-1}\E_{\lambda_n} (v_n^iv_n^j)(Y^{(n)}_k)=\mfB_1(v_\infty^i,v_\infty^j)$,
\item\label{subpoint:B_2}
$(W_{v_n,n},\BBW_{v_n,n})\to_{\lambda_n} (W_{\bv},\BBW_{\bv})$ as $n\to\infty$
in the sense of finite-dimensional distributions, where $W_{\bv}$ is an $\R^m$-valued Brownian motion and
\[
\BBW^{ij}_{\bv}(t)=\int_0^t W_{\bv}^i\mrd W_{\bv}^j+\mfB_2(v^i_\infty,v^j_\infty)t\;.
\]
\end{enumerate}
\end{enumerate}  
\end{assumption}

\begin{rem}\label{rem:stationary_simplification}
As in Remark~\ref{rem:stationary_simplification_n_indep}, under the assumption that $\lambda_n$ is $T_n$-stationary, the simpler bounds
\[
|W_{v_n,n}(k/n)|_{L^{2q}(\lambda_n)} \leq K (k/n)^{1/2}
\quad\text{and} \quad
|\BBW_{v_n,w_n,n}(0,k/n) |_{L^{q}(\lambda_n)} \leq K k/n
\]
for all $0\leq k\leq n$,
imply Assumption~\ref{assump:discrete_W_assump}\ref{point:Kolm_assump}.
Also, Assumption~\ref{assump:discrete_W_assump}\ref{point:fdd_assump}\ref{subpoint:B_1} reduces to $\lim_{n\to\infty}\E_{\lambda_n}(v_n^i w_n^j)=\mfB_1(v_\infty^i,w_\infty^j)$.
\end{rem}

\begin{prop} \label{prop:moments_discrete}
Suppose that Assumption~\ref{assump:discrete_W_assump}\ref{point:Kolm_assump} holds.
Then there exists $K>0$ such that for all $n \in \BBN$, $0 \leq k,\ell\leq n$, and $v,\,w\in C_n^\kappa(\Lambda)$,
\begin{align*}
\Big|W_{v,n}(k/n)-W_{v,n}(\ell/n)\Big|_{L^{2q}(\lambda_n)} &\le K |v|_{C^\kappa} n^{-1/2}|k-\ell|^{1/2}\;,
\\
\Big|\BBW_{v,w,n}(k/n,\ell/n)\Big|_{L^{q}(\lambda_n)} &\le K |v|_{C^\kappa}|w|_{C^\kappa} n^{-1}|k-\ell|\;. \nonumber
\end{align*}
\end{prop}

\begin{proof}
Identical to Proposition~\ref{prop:moments_discrete_n_indep}.
\end{proof}

\begin{prop} \label{prop:auto}
Suppose that Assumption~\ref{assump:discrete_W_assump} holds.
Let $\mfB= \frac12\mfB_1+\mfB_2$.
Then, for 
all $\bv=(v_n)_{n \in \BBN\cup\{\infty\}}$ with
$\lim_{n \to \infty} |v_n-v_\infty|_{C^\kappa} = 0$,
%every $C^\kappa_n(\Lambda,\R^m)$-family $\bv=(v_n)_{n \in \BBN\cup\{\infty\}}$, 
the covariance of $W_{\bv}$ is given by
\[
\E W_{\bv}^i(1)W_{\bv}^j(1)=\mfB(v^i_\infty,v_\infty^j)+\mfB(v^j_\infty,v^i_\infty)\;.
\]
\end{prop}

\begin{proof}
Exactly the same as Proposition~\ref{prop:auto_n_indep}\ref{point:mfB_prime_def_n_indep} upon replacing $W_{v,n}$ by $W_{v_n,n}$ and $W_{v}$ by $W_{\bv}$, and using Assumption~\ref{assump:discrete_W_assump}\ref{point:fdd_assump}\ref{subpoint:B_1}.
\end{proof}

\begin{lem}\label{lem:Lip_sqrt}
Suppose that Assumption~\ref{assump:discrete_W_assump} holds.
Then the symmetric quadratic form
\begin{align*}
\Sigma^{ij}(x) = \mfB(b_\infty^i(x,\cdot),b^j_\infty(x,\cdot)) + \mfB(b_\infty^j(x,\cdot),b_\infty^i(x,\cdot))\;, \quad i,j =1, \dots, d\;,
\end{align*}
is positive semi-definite and the unique positive semi-definite $\sigma$ satisfying $\sigma^2 = \Sigma$ is Lipschitz.
\end{lem}

\begin{proof}
Identical to Lemma~\ref{lem:Lip_sqrt_n_indep}.
\end{proof}

As before, by Lemma~\ref{lem:Lip_sqrt} and~\cite[Cor.~5.1.2]{StroockVaradhan06}, for a Brownian motion $B$ on $\R^d$ and a Lipschitz function $\tilde a : \R^d \to \R^d$, there is a unique strong solution to the SDE
\begin{align}\label{eq:discrete nonprod sde}
\mrd X = \tilde a(X)\mrd t + \sigma(X)\mrd B\;, \quad X(0) = \xi\;.
\end{align}
In particular, the SDE~\eqref{eq:discrete nonprod sde} has uniqueness in law.

\begin{thm}\label{thm:discrete_FS_conv}
Suppose that Assumptions~\ref{assump:conv_drift} and~\ref{assump:discrete_W_assump} hold, and that $\lim_{n \to \infty}\xi_n = \xi \in \R^d$.
Define the c{\`a}dl{\`a}g path
\begin{equation}\label{eq:x_n discrete}
x_{n} : [0,1] \to \R^d\;, \quad x_{n}(t) = X^{(n)}_{\floor{nt}}\;.
\end{equation}
Then $x_n \to_{\lambda_n} X$ in the uniform topology as $n \to \infty$, where $X$ is a
weak solution of the SDE~\eqref{eq:discrete nonprod sde},
where $B$ is a standard Brownian motion in $\R^d$, $\sigma$ is defined as in Lemma~\ref{lem:Lip_sqrt}, and $\tilde a$ is the Lipschitz function given by
\begin{align*}
\tilde a^i(x) = \bar a^i(x) +  \sum_{k=1}^d \mfB_2(b_\infty^k(x,\cdot),\partial_k b_\infty^i(x,\cdot))\;,
\quad i=1,\dots,d\;.
\end{align*}
\end{thm}

\subsection{Homogenization for the illustrative example}
\label{subsec:LSV}

In this subsection, we apply our main result, Theorem~\ref{thm:discrete_FS_conv} in the case where the fast dynamics $T_n$ is a family of intermittent maps as in Section~\ref{subsec:example}.
Using the results from Part~1~\cite{KKMprep}, we verify the hypotheses of
Theorem~\ref{thm:discrete_FS_conv} and deduce convergence to an SDE~\eqref{eq:SDE} with coefficients $\tilde a$ and $\sigma$ as given in~\eqref{eq:coeffs1} and~\eqref{eq:coeffs2}.

Recall that $\Lambda=[0,1]$ and $T_n:\Lambda\to\Lambda$, $n\in\BBN\cup\{\infty\}$, is a family of intermittent maps as in~\eqref{eq:LSV} with parameters
$\gamma_n\in(0,\frac12)$ such that $\lim_{n\to\infty}\gamma_n=\gamma_\infty$.
Let $\mu_n$ be the corresponding family of $T_n$-invariant ergodic absolutely continuous probability measures.
Let $C_n^\kappa(\Lambda)=\{v\in C^\kappa(\Lambda):\int_\Lambda v\mrd \mu_n=0\}$ and fix $q\in(1,\gamma_\infty^{-1}-1)$.  We consider fast-slow systems~\eqref{eq:FS_intro} where $a_n$, $b_n$ satisfy the regularity conditions at the beginning of Subsection~\ref{subsec:general_case} and $b_n\in C_n^{\alpha,\kappa}(\R^d\times\Lambda,\R^d)$.
We require further that $\lim_{n\to\infty}|a_n-a_\infty|_\infty=0$ and that
$a_\infty(x,\cdot):\Lambda\to\R^d$ is H\"older continuous for each fixed $x$.

To apply Theorem~\ref{thm:discrete_FS_conv}, we verify Assumptions~\ref{assump:conv_drift} and~\ref{assump:discrete_W_assump} for appropriate families of probability measures $\lambda_n$.  We do this for the case 
$\lambda_n\equiv\Leb$ using the results in~\cite[Sec.~4.1]{KKMprep}.
The case $\lambda_n=\mu_n$ works in the same way (indeed, this is the easier case in~\cite{KKMprep}).

\begin{prop} \label{prop:A211}
Assumption~\ref{assump:conv_drift} holds with $\bar a(x)=\int_\Lambda a_\infty(x,\cdot)\mrd\mu_\infty$.
\end{prop}

\begin{proof}
Fix $x\in\R^d$ and define $v_n=a_n(x,\cdot)$.
Then $V_n(t)(x)=n^{-1}\sum_{j=0}^{\lfloor nt\rfloor-1}v_n\circ T_n^j$ and it follows from~\cite[Prop.~4.3(a)]{KKMprep} that
$V_n(t)(x)\to_{\Leb} t\int_\Lambda  v_\infty\mrd \mu_\infty
=t\int_\Lambda a_\infty(x,\cdot)\mrd\mu_\infty$.
\end{proof}

\begin{prop} \label{prop:A212i}
Assumption~\ref{assump:discrete_W_assump}\ref{point:Kolm_assump} holds.
\end{prop}

\begin{proof}
Let $p=q+1\in(2,\gamma_\infty^{-1})$ and 
$v = (v_1,\ldots),\,w=(w_1,\ldots) \in \prod_{n \in \BBN} C_n^\kappa(\Lambda)$.
By~\cite[Prop.~4.1]{KKMprep}, there is a constant $C>0$ such that for 
$0\le \ell<k\le n$,
\begin{align*}
|W_{v_n,n}(k/n)-W_{v_n,n}(\ell/n)|_{L^{2q}(\Leb)} &=
n^{-1/2}\Big|\sum_{\ell\le j<k}v_n\circ T_n^j\Big|_{L^{2(p-1)}(\Leb)}
\\ 
&\le Cn^{-1/2}(k-\ell)^{1/2} |v_n|_{C^\kappa}\;,
\\[1ex]
|\BBW_{v_n,w_n,n}(k/n,\ell/n)|_{L^{q}(\Leb)} 
&= 
n^{-1}\Big|\sum_{\ell\le i< j<k}(v_n\circ T_n^i)\otimes(w_n\circ T_n^j)\Big|_{L^{p-1}(\Leb)}
\\
&\le C n^{-1} (k-\ell)
|v_n|_{C^\kappa} |w_n|_{C^\kappa}\;.
\end{align*}
These are the desired estimates.
\end{proof}

\begin{prop} \label{prop:A212ii}
Assumption~\ref{assump:discrete_W_assump}\ref{point:fdd_assump} holds with 
\[
\mfB_1(v,w)=\int_\Lambda vw\mrd\mu_\infty\;,
\qquad
\mfB_2(v,w)=\sum_{\ell=1}^\infty \int_\Lambda
 v\,w\circ T_\infty^\ell \mrd \mu_\infty\;.
\]
\end{prop}

\begin{proof}
Assumption~\ref{assump:discrete_W_assump}\ref{point:fdd_assump}\ref{subpoint:B_1} is verified in~\cite[Prop.~4.3(b)]{KKMprep}.
By~\cite[Prop.~4.2]{KKMprep} together with~\cite[Rem.~2.9]{KKMprep}, 
$(W_{v_n,n},\BBW_{v_n,n})\to_{\Leb} (W_{\bv},\BBW_{\bv})$ 
where $W_{\bv}$ is an $\R^m$-valued Brownian motion and 
\[
\BBW^{ij}_{\bv}(t)=\int_0^t W_{\bv}^i\mrd W_{\bv}^j+ E_\infty^{ij}t\;, \qquad
E_\infty^{ij}=\sum_{\ell=1}^\infty\int_\Lambda v_\infty^i\,v_\infty^j\circ T_\infty^\ell\mrd \mu_\infty\;.
\]
This verifies Assumption~\ref{assump:discrete_W_assump}\ref{point:fdd_assump}\ref{subpoint:B_2}.
\end{proof}

We can now apply Theorem~\ref{thm:discrete_FS_conv}.
Define $\tilde a$ as in~\eqref{eq:coeffs1} and set
\begin{align*}
\Sigma(x) & =
\int_\Lambda b_\infty(x,y)\otimes b_\infty(x,y)\mrd \mu_\infty(y)
\\ & \qquad +
\sum_{\ell=1}^\infty \int_\Lambda \big\{b_\infty(x,y)\otimes b_\infty(x,T_\infty^\ell y)+
b_\infty(x,T_\infty^\ell y)\otimes b_\infty(x,y)\big\}\mrd \mu_\infty(y)\;.
\end{align*}
Let $\sigma$ be the unique positive semidefinite square root of $\Sigma$ as in Lemma~\ref{lem:Lip_sqrt}.
By Theorem~\ref{thm:discrete_FS_conv}, $x_n\to_{\Leb}X$ where $X$ is the unique solution to the SDE~\eqref{eq:SDE} with coefficients $\tilde a$ and $\sigma$.

\section{Banach space valued c{\`a}dl{\`a}g rough paths}
\label{sec:RP}

In this section, we collect all the necessary results on c{\`a}dl{\`a}g rough path theory in Banach spaces which will be needed in the sequel.

For Banach spaces $\MA,\MB$, we denote their algebraic tensor product by
$$
\MA\otimes_a \MB:= \operatorname{span} \left\{ a\otimes b | a \in \MA, b\in \MB \right\}\;.
$$
Given $f\in \MA^*$ (the dual space of $\MA$), $g\in \MB^*,$ one may define an element on $(\MA \otimes_a \MB)^*$ by
$$
(f\otimes g) (\sum_{i=1}^N a_i \otimes b_i) := \sum_{i=1}^N f(a_i) g(b_i)\;.
$$
As a result, we consider $\MA^*\otimes_a \MB^*$ as a subspace of $(\MA \otimes_a \MB)^*$. Generally, there are different (inequivalent) norms on $\MA\otimes_a \MB$. We call a norm $|\cdot |_{\MA\otimes \MB}$ on the vector space $\MA\otimes_a \MB$ admissible (or reasonable), if for any $a\in \MA, b\in \MB, f\in \MA^*, g\in \MB^*,$
\begin{equation}\label{eq:tensor_norm_admissible}
|a \otimes b |_{\MA\otimes \MB} \leq |a|_\MA |b|_{\MB}\;, \quad |f\otimes g|_{(\MA\otimes \MB)^*} \leq |f|_{\MA^*} |g|_{\MB^*}\;,
\end{equation}
where $|\cdot |_{(\MA\otimes \MB)^*}$ is defined as the dual norm on $(\MA \otimes_a \MB, |\cdot |_{\MA\otimes \MB})^*$.
Examples of admissible norms are the projective tensor norm and the injective tensor norm, see~\cite[Sec.~6.1]{Ryan02}. 
One may then complete $\MA \otimes_a \MB$ under $|\cdot |_{\MA\otimes \MB}$ to obtain a Banach space.
All the tensor product spaces $\MA \otimes \MB$ we consider in the sequel will implicitly be assumed to be Banach spaces completed by such an admissible norm.

\begin{defi}
    A partition over an interval $[s,t]$ is a set $\MP$ of subintervals of $[s,t]$ of the form $\MP = \{[t_0,t_1],[t_1,t_2],\ldots, [t_{k-1},t_{k}]\}$ with $t_i < t_{i+1}$ and $t_0=s$, $t_k=t$.
We define the mesh size of the partition as $|\op| := \max_{[u,v]\in\op} |u-v|$.

For a Banach space $\MB$ and $p > 0$, let $\MV^{\var p}([s,t],\MB)$ denote the space of all functions $\Xi : \{(u,v) \in [s,t]^2 \ssep u \leq v \}\to \MB$ such that $\Xi(u,u) = 0$ and
\[
\|\Xi\|_{\var p;[s,t]} :=  \sup_{\MP} \Big(\sum_{[u,v]\in \MP} |\Xi(u,v)|^p\Big)^{1/p} < \infty\;,
\]
where the supremum is over all partitions of $[s,t]$.
\end{defi}

Note that if $p \geq 1$, then $\MV^{\var p}([s,t],\MB)$ is a Banach space with norm $\|\cdot\|_{\var p;[s,t]}$.
In the sequel, we will drop the reference to the interval $[s,t]$ whenever $[s,t]=[0,T]$.
We will also occasionally refer to $p$-variation over not necessarily closed intervals, i.e., $(s,t]$ or $[s,t)$ instead of $[s,t]$, with the obvious interpretation.

For a Banach space $\MB$, we equip $\MB \oplus \MB^{\otimes 2}$ with the multiplication operation $(a,M) (b,N) := (a + b, M + a \otimes b + N)$.
Note that the multiplicative identity in $\MB\oplus\MB^{\otimes 2}$ is $(0,0)$ and every element posses an inverse given by $(a,M)^{-1}= (-a, -M + a\otimes a)$.
Hence $\MB\oplus\MB^{\otimes 2}$ is a group.

\begin{defi}\label{def:RP}
Let $\MB$ be a Banach space.
For a path $\BX : [s,t] \to \MB\oplus\MB^{\otimes 2}$ and $s\leq u \leq v \leq t$, define the increment $\BX(u,v):=(X(u,v), \X(u,v)):=\BX(u)^{-1}\BX(v)$.
For $p \geq 1$, define the (homogeneous) $p$-variation of $\BX$ by
\begin{equation*}
\| \BX \|_{\var p;[s,t]} := \| X \|_{\var p;[s,t]} +   \| \X \|^{1/2}_{\var{p/2};[s,t]} \; .
\end{equation*}
For $p \in [2,3)$, a $p$-rough path over $\MB$ is a c{\`a}dl{\`a}g function $\BX:[0,T] \to \MB\oplus\MB^{\otimes 2}$ such that $\BX(0) = 0$ and $\|\BX\|_{\var p} < \infty$.
For $p$-rough paths $\BX,\tilde\BX$, we define the (inhomogeneous) rough path metric by 
\begin{equation} \label{equ:p-var_RP_dist}
\| \BX ; \tilde \BX  \|_{\var p} := \| X - \tilde X \|_{\var p} +   \| \BBX -\tilde \BBX \|_{\var{p/2}}  \; ,
\end{equation}
as well as the (Skorokhod-type) $p$-variation metric
\begin{equation}\label{eq:sigma_p-var_def}
\sigma_{\var p}(\BX,\tilde\BX) := \inf_{\llambda \in \LLambda} \big\{ |\llambda| + \|\BX;\tilde\BX\circ\llambda\|_{\var p} \big\}\;,
\end{equation}
where $\LLambda$ denotes the set of all continuous increasing bijections $\llambda:[0,T] \to [0,T]$, and
\[
|\llambda| := \sup_{t\in[0,T]} |t-\llambda(t)|\;.
\]

Let $\MD^{\var p}(\MB)$ denote the space of all $p$-rough paths equipped with the metric $\sigma_{\var p}$.
For $p \in [1,2)$ define the $p$-variation $\|\cdot\|_{\var p}$ of a path $X : [0,T] \to \MB$, as well as the metric $\sigma_{\var p}$ and space $D^{\var p}(\MB)$ in the exact same way as above but without the component $\BBX$.
\end{defi}

The purpose of the metric $\sigma_{\var p}$
is to provide convenient tightness results.
In short, tightness in the metric space $(D^{\var p},\sigma_{\var p})$ is implied by tightness of $p'$-variation for $p'<p$ and tightness in the ($J_1$) Skorokhod space,
with the latter two being simpler to check;
see the proofs of Lemmas~\ref{lem:V_tight_discrete} and~\ref{lem:discrete_rp_tightness}.
Likewise for $\MD^{\var p}$.
The same is not true if we replace $\sigma_{\var p}$ by $\|\cdot;\cdot\|_{\var p}$.

We next state a basic interpolation estimate which will be helpful in the sequel.
Define
\[
\|\BX;\tilde\BX\|_\infty = \|X-\tilde X\|_\infty + \|\BBX - \tilde\BBX\|_\infty\;,
\]
where $\|\Xi\|_\infty := \sup_{s,t} |\Xi(s,t)|$ (as usual, we treat $X$ as a two parameter function by $X(s,t) = X(t)-X(s)$).
\begin{lem}
For $p'\geq p\geq 1$ and $\BX,\tilde\BX : [0,T] \to \MB\oplus\MB^{\otimes 2}$, it holds that
\begin{equation}\label{eq:interpolation}
\|\BX;\tilde\BX\|_{\var{p'}}
\leq \|\BX;\tilde\BX\|_\infty^{1-p/p'}\|\BX;\tilde\BX\|_{\var{p}}^{p/p'}\;.
\end{equation}
\end{lem}

\begin{proof}
We readily see that
\begin{equation*}
\|\BX;\tilde\BX\|_{\var{p'}}
\leq \|X-\tilde X\|_\infty^{1-p/p'} \| X-\tilde X \|_{\var p}^{p/p'} + \|\BBX - \tilde\BBX\|_\infty^{1-p/p'}\|\BBX - \tilde\BBX\|_{\var{p/2}}^{p/p'}\;,
\end{equation*}
and the conclusion follows by H{\"o}lder's inequality $a^\theta\bar a^{1-\theta} + b^\theta\bar b^{1-\theta} \leq (a+b)^\theta(\bar a+\bar b)^{1-\theta}$ for $\theta \in [0,1]$ and $a,\bar a,b,\bar b \geq 0$.
\end{proof}

We now introduce rough integration in the level-$2$ rough path case.
Given Banach spaces $\MB, \ME$, let $\BL(\MB ,\ME)$ denote the space of bounded linear operators from $\MB$ to $\ME$.
For $p \in [2,3)$ and $X\in D^{\var p}(\MB)$, we call $(Y,Y')$ an $\ME$-valued $X$-controlled rough path if
\[
Y\in D^{\var p}(\ME)\;, \quad Y' \in D^{\var p}(\BL(\MB,\ME))\;,
\]
and $R \in \MV^{\var{p/2}}(\ME)$, where
\begin{equation}\label{eq:def_R}
R(s,t) := Y(s,t) - Y'(s) X(s,t)\;.
\end{equation}
We denote the space of $X$-controlled rough paths as $\MD^{\var{p/2}}_X(\ME)$.
In the following, we are interested in $\R^d$-valued RDEs, i.e. $\ME=\R^d$.
In this case, one has the following stability of rough integration.
Remark that, $\MB^* \otimes_a \MB^*$ is a subspace of $(\MB \otimes \MB)^*$ by admissibility of norms~\eqref{eq:tensor_norm_admissible}
and therefore $f\otimes g \in (\MB \otimes \MB)^*$ for every $f,g\in\MB^*$.
In particular, $\Xi$ in the statement of the following lemma is well-defined.

\begin{lem}\label{lem:well-defined-mapping}
Let $\BX\in\MD^{\var p}(\MB)$, $(Y,Y') \in \MD^{\var{p/2}}_X(\R^d)$, and $H \in C^2(\R^d, \BL(\MB, \R^d))$.
Then, for every $t\in[0,T]$, the following integral (with values in $\R^d$) is well-defined
\begin{equation}\label{eq:rough_int_def}
\I_{\BX}(Y)(t):=\int_0^t H(Y^-(s)) \mrd \BX(s):= \lim_{|\op|\rightarrow 0} \sum_{[u,v]\in \op} \Xi(u,v)\;,
\end{equation}
where $\op$ are partitions of $[0,t]$ and, for $i=1,...,d$ and $0\leq u\leq v\leq T$,
\begin{equation*}
\Xi(u,v)^i=H^i(Y(u))X(u,v) + \sum_{k=1}^d \left( \partial_k H^i(Y(u)) \otimes (Y'(u))^k \right)\X(u,v)\;.
\end{equation*}
Furthermore, $(H(Y), DH(Y)Y')$ and $(\I_{\BX}(Y), H(Y))$ are $X$-controlled rough paths.
\end{lem}

\begin{proof}
%First, we check that $\Xi(u,v) \in \R^d$ is well-defined.
%Indeed, for $i,k=1,\ldots, d$, $H^i(Y(u))$, $\partial_k H^i(Y(u))$, and $(Y(u)')^k$ are elements of $\MB^*$.
%By admissibility of norms~\eqref{eq:tensor_norm_admissible}, $\MB^* \otimes_a \MB^*$ is a subspace of $(\MB \otimes \MB)^*$, and thus $ \partial_k H^i(Y(u))\otimes (Y'(u))^k \in (\MB \otimes \MB)^*$.
%Hence $\Xi(u,v)$ is well-defined as claimed.

The claim that $(H(Y), DH(Y)Y')$ is an $X$-controlled rough path follows from Taylor expansion.
Indeed, defining
$$
R^{H(Y)}(s,t):= H(Y(t))-H(Y(s))-DH(Y(s)) Y'(s) X(s,t)\;,
$$
one can check that $R^{H(Y)}\in \MV^{\frac p 2}(\BL(\MB, \R^d))$.
Then one has the identity
$$
\Xi(s,t)-\Xi(s,u)-\Xi(u,t)= -R^{H(Y)}(s,u) X(u,t) - \left( DH(Y) Y'\right)(s,u) \X(u,t)\;.
$$
According to the generalized sewing lemma~\cite[Thm.~2.5]{FZ18}, the integral $\I_{\BX}(Y)$ is well-defined, and furthermore one has the local estimate
\begin{multline*}
|\I_{\BX}(Y)(s,t) - \Xi(s,t)| \leq C \big[\|R^{H(Y)}\|_{\var{p/2};[s,t)} \|X\|_{\var p;(s,t]}
\\
+ \|DH(Y)Y'\|_{\var p;(s,t]} \|\X\|_{\var{p/2};[s,t)}\big]\;,
\end{multline*}
which implies that $(\I_{\BX}(Y), H(Y))$ is also an $X$-controlled rough path.
\end{proof}

\begin{rem}\label{admissibility}
    Generally, to integrate $(Y,Y')$ against $\BX$, one needs $Y(t) \in \BL(\MB, \ME)$ and $Y'(t) \in \BL(\MB, \BL(\MB, \ME))$ to have the identity $Y(s,t)=Y'(s) X(s,t) + R(s,t)$.
In this case, one further needs the embedding $\BL(\MB, \BL(\MB, \ME)) \hookrightarrow \BL(\MB \otimes \MB, \ME)$ to define $\Xi(s,t):= Y(s) X(s,t) + Y'(s) \X(s,t)$.
Luckily, in the above case where $\ME=\R^d$, the embedding assumption is replaced by the fact $DH(Y)Y' \in \BL(\MB \otimes \MB, \R^d)$ which follows by admissibility of norms.
\end{rem}

The main convergence result for rough differential equations which we will require is the following.
The proof, which we omit, is essentially the same as the finite dimensional case, i.e.,~\cite[Thm.~3.8,~3.9]{FZ18}, thanks to admissibility of norms.

\begin{thm} \label{thm:RDE}
Let $\MA,\MB$ be Banach spaces, $p \in [2,3)$, $q \in [1,p/2]$, and $F\in C^\beta(\R^d,\BL(\MA,\R^d)), H \in C^\gamma(\R^d,\BL(\MB,\R^d))$ for $\beta>q, \gamma > p$.
Then, for any $V \in D^{\var q}(\MA)$, $\BX \in \MD^{\var p}(\MB)$, and $Y_0 \in \R^d$, there exists a unique $X$-controlled rough path $(Y,Y') \in \MD_X^{\var{p/2}}(\R^d)$ solving the equation
\begin{equation}\label{equ:jRDE}
Y(t)= Y_0 + \int_0^t F(Y^-(s)) \mrd V(s) + \int_0^t H(Y^-(s)) \mrd \BX(s) \; .
\end{equation}
Moreover, the solution map is locally Lipschitz in the sense that
\begin{equation}
\|Y-\tY\|_{\var p} \lesssim \| \BX ;\tilde \BX \|_{\var p} + \|V-\tilde{V} \|_{\var q} + |Y_0-\tY_0|\;,
\end{equation}
where the proportionality constant is uniform over bounded sets of driving signals. 
\end{thm}

Recall that in~\eqref{equ:jRDE}, $\int_0^t H(Y^-(s)) \mrd \BX(s)$ is defined by~\eqref{eq:rough_int_def}
and that $\int_0^t F(Y^-(s)) \mrd V(s)$ is the classical Young integral~\cite[Prop.~2.4]{FZ18}
\[
\int_0^t F(Y^-(s)) \mrd V(s) = \lim_{|\op|\rightarrow 0} \sum_{[u,v]\in \op} 
F(Y(u))V(u,v)\;,
\]
where $\op$ are partitions of $[0,t]$, which is well-defined since $1/q+1/p>1$.
Note that the restriction $q\leq p/2$ arises from the Young estimate $\|\int_0^\cdot (Z^-(s)-Z(0))\mrd V(s)\|_{\var q}\lesssim \|Z\|_{\var p}\|V\|_{\var q}$
and the requirement that $R \in \MV^{\var{p/2}}(\R^d)$ in the definition~\eqref{eq:def_R}.

For our purposes, it will be useful to record the following corollary stated in terms of the metrics $\sigma_{\var p}$ and $\sigma_{\var q}$.

\begin{cor}\label{cor:sigmas}
Let notation be as in Theorem~\ref{thm:RDE}.
Consider the solution map to equation~\eqref{equ:jRDE}
\begin{align*}
\Phi &: D^{\var q}(\MA) \times \MD^{\var p}(\MB) \times \R^d \to D^{\var p}(\R^d)\;,
\\
\Phi &: (V,\BX,Y_0) \mapsto Y\;.
\end{align*}
Let $p'>p$ and equip $D^{\var p}(\R^d)$ with the norm $|Y(0)|+\|Y\|_{\var{p'}}$ and $D^{\var q}(\MA) \times \MD^{\var p}(\MB) \times \R^d$ with the product metric $(\sigma_{\var q},\sigma_{\var p}, |\cdot|)$.
Then every point $(V,\BX,Y_0)$, where $V,\BX$ are continuous, is a continuity point of $\Phi$.
\end{cor}

\begin{proof}
It suffices to consider $p' \in (p,\gamma)$.
Let $\BX \in \MD^{\var p}(\MB)$ be continuous.
We claim that $\sigma_{\var p}(\BX_n, \BX) \to 0$ implies $\|\BX_n; \BX\|_{\var{p'}} \to 0$.
Indeed,  $\sigma_{\var p}(\BX_n, \BX) \to 0$ implies the existence of $\{\llambda_n\}_{n \geq 1}\subset \LLambda$ satisfying both $|\llambda_n|\to0$ and $\|\BX_n;\BX\circ\llambda_n\|_{\var p}\to 0$.
Observe that $\|\BX;\BX\circ\llambda_n\|_\infty \to 0$ by continuity of $\BX$,
and therefore, combining with the interpolation estimate~\eqref{eq:interpolation},
$\|\BX;\BX\circ\llambda_n\|_{\var{p'}} \to 0$.
Since $\|\BX_n;\BX\|_{\var{p'}}\leq \|\BX_n;\BX\circ\llambda_n\|_{\var{p'}}+\|\BX\circ\llambda_n;\BX\|_{\var{p'}}$,
this proves the claim.
The same considerations apply to continuous $V\in D^{\var q}(\MA)$ and $q'>q$.
The result follows from Theorem~\ref{thm:RDE} by taking $q' \in (q,\beta\wedge p')$.
\end{proof}

\begin{rem}
Recall that, for the classical ($J_1$) Skorokhod space $D$, a pair $(x,y) \in D^2$ is a continuity point of the addition map $D^2 \to D$, $(x,y) \mapsto x+y$, whenever one of $x$ or $y$ is continuous.
In a similar way, if one instead equips $D^{\var p}(\R^d)$ with the metric $|Y(0)-\bar Y(0)|+\sigma_{\var{p'}}(Y,\bar Y)$, then one can show that $(V,\BX,Y_0)$ is a continuity point of $\Phi$ whenever one of $\BX$ or $V$ is continuous.
\end{rem}

We conclude this section with the following result which will be helpful in controlling the $p$-variation and c{\`a}dl{\`a}g modulus of continuity of paths.

\begin{prop}\label{prop:piecewise_constant_var_bound}
Let $(\mathcal{O},\mathcal{F},\mathbb{P})$ be a probability space and let $\{\BX_t\}_{t\in[0,T]}=\{(X_t,\BBX_t)\}_{t\in[0,T]}$
be a 
$\MB\oplus \MB^{\otimes 2}$-valued stochastic process defined on $(\mathcal{O},\mathcal{F},\mathbb{P})$.
Suppose further that, for $\mathbb{P}$-a.e. $o\in\mathcal{O}$,
$t\mapsto \BX^o_t$ is c{\`a}dl{\`a}g, piecewise constant, and has
jump times contained in a deterministic set $\{t_j\}_{0 \leq j \leq n} \subset [0,T]$ with $0=t_0<t_1<\ldots < t_n=T$, such that, for some $C_1,C_2>0$, $\beta \in (0,\frac12]$, and $q \in [1,\infty]$,
\[
| X(t_i,t_{j})|_{L^{2q}(\mathbb{P})} \leq C_1|t_{j}-t_i|^{\beta}\;, \quad | \BBX(t_i,t_{j}) |_{L^q(\mathbb{P})} \leq C_2|t_{j}-t_i|^{2\beta}\;.
\]
If $2q>\frac1{\beta}$, then for any $\alpha \in (\frac1{2q},\beta)$
\begin{equation}\label{eq:piecewise_const_var_bound}
\E[\|\BX\|_{\var{1/\alpha}}^{2q}]^{\frac1{2q}}
\leq C T^{\alpha-\frac1{2q}}(C_1+C_2^{1/2})
\end{equation}
and
\begin{equation}\label{eq:piecewise_const_Hol_bound}
\E\Big[\Big|\sup_{t_i \neq t_j} \frac{|X(t_i,t_j)| + |\BBX(t_i,t_j)|^{1/2}}{|t_i-t_j|^{\alpha-\frac1{2q}}} \Big|^{2q}\Big]^{\frac1{2q}} \leq C (C_1+C_2^{1/2})
\end{equation}
for a constant $C > 0$ depending only on $\alpha,\beta,q$.
\end{prop}

For the proof, we require the following lemma.

\begin{lem}\label{lem:cont_path_estimate}
Let $\BX$ be as in Proposition~\ref{prop:piecewise_constant_var_bound}.
Then there exists a $\mathbb{P}$-a.s. continuous $\MB\oplus \MB^{\otimes 2}$-valued process $\{\tilde \BX_t\}_{t\in[0,T]} = \{(\tilde X_t,\tilde \BBX_t)\}_{t\in[0,T]}$ such that $\BX(t_i) = \tilde\BX(t_i)$ for all $i=0,\ldots, n$, and
\begin{equation}\label{eq:cont_path_bounds}
| \tilde X(s,t) |_{L^{2q}(\mathbb{P})} \leq 3^{1-\beta}C_1|t-s|^{\beta}\;, \quad | \tilde\BBX(s,t) |_{L^q(\mathbb{P})} \leq 3^{2-2\beta}(C_2 + C_1^2)|t-s|^{2\beta}\;.
\end{equation}
\end{lem}

\begin{proof}
Let us define $(\tilde X,\tilde\X)$ for $t \in [t_j,t_{j+1})$ by
\begin{eqnarray*}
\tilde X(t)&:=&X(t_j) + \frac{t-t_j}{t_{j+1}-t_j} X(t_j, t_{j+1})\;, \\
\tilde{\X}(0,t)&:=& \X(0,t_j)+ \frac{t-t_j}{t_{j+1}-t_j}\left(\X(0,t_{j+1})- \X(0,t_j)\right)\;.
\end{eqnarray*}
To prove~\eqref{eq:cont_path_bounds}, consider $s<t$ with $s\in [t_j,t_{j+1})$, $t\in[t_{k},t_{k+1})$.
Further we suppose that $j<k$ (the case $j=k$ is similar and simpler).
Then
\begin{align*}
|\tilde X(s,t)|_{L^{2q}(\mathbb{P})} &\leq |\tilde X(s,t_{j+1})|_{L^{2q}(\mathbb{P})} + |\tilde X(t_{j+1},t_{k})|_{L^{2q}(\mathbb{P})} + |\tilde X(t_{k},t)|_{L^{2q}(\mathbb{P})}
\\
&\leq C_1(|t_{j+1}-s|^\beta + |t_{j+1}-t_k|^\beta + |t-t_k|^\beta)
\\
&\leq 3^{1-\beta}C_1|t-s|^\beta\;.
\end{align*}
Furthermore, one can check that 
\begin{align*}
\tilde \BBX(s,t_{j+1}) = \frac{t_{j+1}-s}{t_{j+1}-t_j} \BBX(t_j,t_{j+1}) + \frac{(t_{j+1}-s)(s-t_j)}{(t_{j+1}-t_j)^2}X_{t_j,t_{j+1}}^{\otimes 2}\;,
\end{align*}
from which it follows that
\[
|\tilde \BBX(s,t_{j+1})|_{L^q(\mathbb{P})} \leq (C_2 + C_1^2)|t_{j+1}-s|^{2\beta}\;.
\]
A similar estimate holds for $\tilde \BBX(t_k,t)$.
Hence
\begin{align*}
|\tilde \BBX(s,t)|_{L^q(\mathbb{P})}
&\leq |\tilde \BBX(s,t_{j+1})|_{L^q(\mathbb{P})} + |\tilde \BBX(t_{j+1},t_k)|_{L^q(\mathbb{P})} + |\tilde \BBX(t_{k},t)|_{L^q(\mathbb{P})}
\\
&\quad + |\tilde X(s,t_{j+1}) \otimes \tilde X(t_{j+1},t_k)|_{L^q(\mathbb{P})} + |\tilde X(s,t_k) \otimes X(t_k,t)|_{L^q(\mathbb{P})}
\\
&\leq 3^{1-2\beta}(C_2 + C_1^2)|t-s|^{2\beta} + C_1^2 |t-s|^{2\beta} + 2^{1-\beta}C_1^2 |t-s|^{2\beta}
\\
&\leq 3^{2-2\beta}(C_2 + C_1^2)|t-s|^{2\beta}\;.
\end{align*}
\end{proof}

\begin{proof}[Proof of Proposition~\ref{prop:piecewise_constant_var_bound}]
Let $\tilde \BX$ be as in Lemma~\ref{lem:cont_path_estimate} and suppose $2q>\frac1\beta$ and $\alpha \in (\frac1{2q},\beta)$.
Using the notation in Appendix~\ref{appendix:Besov}, we have by Corollary~\ref{cor:BesovVarEmbedding}
\begin{align*}
\E[\|\tilde \BX \|_{\var{1/\alpha}}^{2q}]^{\frac{1}{2q}}
&\leq C(\alpha,q)T^{\alpha-\frac1{2q}}\E[\|\tilde \BX\|^{2q}_{W^{\alpha,2q}}]^{\frac{1}{2q}}
\\
&= C(\alpha,q)T^{\alpha-\frac1{2q}} \E\Big[ \iint_{[0,T]^2} \frac{| \tilde X(s,t)|_{\MB}^{2q} + | \tilde \BBX(s,t)|_{\MB^{\otimes 2}}^{q}}{|t-s|^{2\alpha q + 1}}\mrd s \mrd t\Big]^{\frac{1}{2q}}\;.
\end{align*}
Using the estimate~\eqref{eq:cont_path_bounds} and the condition $\alpha<\beta$, the final expectation is bounded by $\lambda(C_1+C_2^{1/2})$, where $\lambda$ depends only on $\beta-\alpha$.
In exactly the same way, using Corollary~\ref{cor:BesovHolderEmbedding}, $\E[\|\tilde\BX\|_{\Hol{(\alpha-\frac1{2q})}}^{2q}]^{\frac1{2q}} \leq C (C_1+C_2^{1/2})$.
The conclusion follows since $\tilde \BX(t_i) = \BX(t_i)$ and $\BX$ is constant on $[t_i,t_{i+1})$.
\end{proof}

\section{Proof of the main result}
\label{sec:discrete_proof}

This section is devoted to the proof of Theorem~\ref{thm:discrete_FS_conv}.
Throughout this section, we let notation be as in Section~\ref{subsec:general_case}.

The first step is to reformulate the system~\eqref{eq:FS_intro} as a c{\`a}dl{\`a}g controlled ODE.
Let us fix $\kappa' \in (0,\bar\kappa)$ and $\theta \in (2,\alpha-\frac dq)$, and introduce the Banach spaces
\begin{equation*}
\MA = {C}^{1+\kappa'}(\R^d,\R^d) \quad \text{ and } \quad \MB = {C}^{\theta}(\R^d,\R^d)\;.
\end{equation*}
We furthermore equip $\MB^{\otimes 2}$ with the admissible norm as specified in~\cite[Prop.~4.5]{KM17}.

For any $\eta \geq 0$, it holds for the point evaluation map $F : \R^d \to \BL(C^\eta(\R^d,\R^d),\R^d)$, given by $F(x) : u \mapsto u(x)$, that $F \in C^\eta(\R^d, \BL(C^\eta(\R^d,\R^d),\R^d))$.
We let $F : \R^d \to \BL(\MA,\R^d)$ and $H : \R^d \to \BL(\MB,\R^d)$ denote the corresponding point evaluation maps.

The following lemma is now immediate from Theorem~\ref{thm:RDE}.
\begin{lem}\label{lem:nonprod RDE well-posed}
The c{\`a}dl{\`a}g RDE
\begin{equation}\label{eq:disc_RDE}
\mrd x(t) = F(x^-(t))\mrd V(t) + H(x^-(t))\mrd \BW(t)\;, \quad x(0) = \xi \in \R^d
\end{equation}
is well-posed for any $(V,\BW) \in D^{\var\beta}([0,1],\MA)\times \MD^{\var p}(\MB)$ with $\beta \in [1,1+\kappa')$ and $p \in [2,\theta)$
such that $\beta\leq p/2$.
\end{lem}

We introduce the $\MA$-valued and $\MB$-valued paths
\[
V_n(t) = n^{-1}\sum_{k=0}^{\floor{tn}-1} a_n(\cdot, Y^{(n)}_{k})\;, \quad W_n(t) = n^{-1/2}\sum_{k=0}^{\floor{tn}-1} b_n(\cdot,Y^{(n)}_{k})\;,
\]
and let $\BW_n = (W_n,\BBW_n)$ be the canonical level-$2$ lift of $W_n$ given by
\[
\BBW_n(t) = \int_0^t W_n(r)\otimes\mrd W_n(r)\;.
\]
Remark that $\BW_n$ is a $p$-rough path over $\MB$ for any $p\in[2,3)$ in the sense of Definition~\ref{def:RP}.

\begin{lem}\label{lem:discont_ODE_consistent}
The path $x_n$ given by~\eqref{eq:x_n discrete} is the unique solution of the c{\`a}dl{\`a}g ODE
\begin{equation}\label{eq:discrete nonprod ODE}
\mrd x_n = F(x_n^-)\mrd V_n + H(x_n^-)\mrd W_n\;, \quad x_n(0) = \xi_n \in \R^d\;.
\end{equation}
\end{lem}

\begin{proof}
Observe that $x_n$ given by~\eqref{eq:x_n discrete} satisfies for all $1 \leq k \leq n$
\begin{align*}
x_n(k/n) - x_n((k-1)/n)
&= n^{-1}a(x_n((k-1)/n), Y^{(n)}_{k-1})
\\
&\quad + n^{-1/2}b(x_n((k-1)/n), Y^{(n)}_{k-1})
\\
&= \int_{(k-1)/n}^{k/n} F(x^-_n(s)) \mrd V_n(s) + H(x^-_n(s))\mrd W_n(s) \;.
\end{align*}
\end{proof}

Following Lemmas~\ref{lem:nonprod RDE well-posed} and~\ref{lem:discont_ODE_consistent}, we are reduced to showing convergence in law for $V_n$ and $\BW_n$ in suitable rough path topologies and identifying the solution of the limiting RDE with an SDE.
We first establish this result for the case that the support of $a_n$ and $b_n$ is uniformly bounded, i.e., there exists a compact set $\mathfrak{K}\subset \R^d$ such that the support of $a_n$ and $b_n$ is contained in $\mathfrak{K}\times \Lambda$ for all $n \in \BBN\cup\{\infty\}$.

\begin{thm}\label{thm:RP_discrete_conv}
Suppose that Assumptions~\ref{assump:conv_drift} and~\ref{assump:discrete_W_assump} hold and that the support of $a_n$ and $b_n$ is uniformly bounded.
Then, for any $p \in (2,3)$ and $\beta \in (1,2)$, there exists a random variable $(V,\BW)$ in $D^{\var\beta}(\MA) \times \MD^{\var p}(\MB)$ such that $(V_n, \BW_n) \to_{\lambda_n} (V,\BW)$, and such that $(V,\BW)$ is a.s.\ continuous.
Moreover, if $\beta \in (1, 1+\kappa')$, $p \in (2,\theta)$, and $\beta\leq p/2$, then the RDE~\eqref{eq:disc_RDE} driven by $(V,\BW)$ along the vector fields $(F,H)$ is a weak solution of the SDE~\eqref{eq:discrete nonprod sde}.
\end{thm}

Before proving Theorem~\ref{thm:RP_discrete_conv}, we first state an immediate consequence of Corollary~\ref{cor:sigmas}, Lemma~\ref{lem:discont_ODE_consistent}, Theorem~\ref{thm:RP_discrete_conv}, and the continuous mapping theorem.

\begin{cor}\label{cor:conv_b_compact_supp}
Suppose we are in the setting of Theorem~\ref{thm:discrete_FS_conv} and that the support of $a_n$ and $b_n$ is uniformly bounded.
Then, for any $p>2$, $x_n \to_{\lambda_n} X$ in the $p$-variation norm $|x(0)| + \|x\|_{\var p}$, where $X$ is a weak solution of the SDE~\eqref{eq:discrete nonprod sde}. 
\end{cor}

We break the proof of Theorem~\ref{thm:RP_discrete_conv} into several lemmas.

\begin{lem}\label{lem:V_tight_discrete}
Suppose that Assumption~\ref{assump:conv_drift} holds and that the support of $a_n$ is uniformly bounded.
Then for every $\beta >1$
\[
\|V_n - V\|_{\var\beta} \to_{\lambda_n} 0\;,
\]
where $V \colon [0,1] \to \MA$ is the deterministic path $V(t) = t\bar a$.
\end{lem}

\begin{proof}
Let $\mathfrak{K}\subset\R^d$ be compact such that $\mathfrak{K}\times \Lambda$ contains the support of $a_n$.
Then the embedding $C^{1+\bar\kappa}(\mathfrak{K},\R^d) \hookrightarrow \MA$ is compact.
Observe further that,
for all $s,t\in[0,1]$ with $|t-s|>n^{-1}$,
\begin{equation}\label{eq:1var_V}
|V_n(t) - V_n(s)|_{C^{1+\bar\kappa}} \leq 2|t-s| |a_n|_{C^{1+\bar\kappa,0}}\;.
\end{equation}
It follows that $(V_n)_{n \geq 1}$ satisfies the compact containment condition~\cite[Rem.~3.7.3]{EK05}
and condition~\cite[Thm.~3.7.2(b)]{EK05}.
Hence, by the tightness criterion~\cite[Thm.~3.7.6]{EK05},
$(V_n)_{n \geq 1}$ is tight in the ($J_1$) Skorokhod space $D([0,1],\MA)$ (note that~\cite[Thm.~3.7.6]{EK05} implies only relative compactness, but tightness is a consequence of the proof).

Observe next that, for $\beta>1$, the interpolation estimate~\eqref{eq:interpolation} implies $\sigma_{\var\beta}(X,Y) \leq \sigma_\infty(X,Y)^{1-1/\beta}(\|X\|_{\var 1}+\|Y\|_{\var 1})^{1/\beta}$,
where $\sigma_\infty$ is the ($J_1$) Skorokhod metric defined as in~\eqref{eq:sigma_p-var_def} with $\|\cdot;\cdot\|_{\var p}$ replaced by $\|\cdot;\cdot\|_\infty$.
Moreover, the map $X\mapsto \|X\|_{\var 1}$ is invariant under reparametrizations and is lower semi-continuous under the metric $\|\cdot;\cdot\|_\infty$ and thus under the metric $\sigma_\infty$.
It follows that, for every $R\geq 0$ and every compact subset $K$ of the classical ($J_1$) Skorokhod space $D([0,1],\MA)$, the set
\[
\{X\in K \ssep \|X\|_{\var 1} \leq R\}
\]
is compact in $D^{\var \beta}(\MA)$.
Furthermore,~\eqref{eq:1var_V} implies that, as a c{\`a}dl{\`a}g path with values in $C^{1+\bar \kappa}(\R^d,\R^d)$, $V_n$ has $1$-variation uniformly bounded in $n \geq 1$ and $Y^{(n)}_0\in \Lambda$.
It follows that $(V_n)_{n \geq 1}$ is tight in $D^{\var \beta}(\MA)$.
Hence, by Prokhorov's theorem~\cite[Thm.~5.1]{Billingsley}, $(V_n)_{n \geq 1}$ is weakly relatively compact in the space of probability measures on $D^{\var \beta}(\MA)$,
and by Assumption~\ref{assump:conv_drift}, the only possible limit point is $V$.
Finally, since $V$ is continuous and deterministic, the same argument as in the proof of Corollary~\ref{cor:sigmas}
implies that $\|V_n - V\|_{\var\beta} \to_{\lambda_n} 0$ for all $\beta>1$.
\end{proof}

Showing convergence of $W_n$ is more involved.

\begin{lem}\label{lem:unif_W_bound_discrete}
Suppose that Assumption~\ref{assump:discrete_W_assump}\ref{point:Kolm_assump} holds and that the support of $b_n$ is uniformly bounded.
Then
\begin{align*}
\E_{\lambda_n}\Big[| W_n(k/n)-W_n(\ell/n)|_{\MB}^{2q}\Big]^{1/(2q)}
&\lesssim n^{-1/2}|k-\ell|^{1/2}\;,
\\
\E_{\lambda_n}\Big[| \BBW_n(k/n,\ell/n)|_{\MB^{\otimes 2}}^{q}\Big]^{1/q}
&\lesssim n^{-1}|k-\ell|\;,
\end{align*}
uniformly in $n \geq 1$ and $0 \leq k,\ell\leq n$.
\end{lem}

\begin{proof}
For a function $u : \R^d \to \R$, let us introduce the notation
\[
\Delta_\sigma u(x) = u(x+\sigma)-u(x)\;, \quad \textnormal{ and } \quad \Delta^{m+1}_\sigma = \Delta_\sigma \circ \Delta_\sigma^m\;.
\]
For $s > 0$ and $p \geq 1$, recall the Besov space $B^s_p$ consisting of all $L^p$ functions $u : \R^d \to \R$ such that
\[
|u|_{B^s_p}^p = |u|_{L^p} + \int_{|\sigma|\leq 1} |\sigma|^{-sp-d} |\Delta_\sigma^{\roof s +1} u|^p_{L^p} \mrd \sigma < \infty\;.
\]
Let us further introduce the notation
\[
\Delta^m_\sigma W_n(k/n,\ell/n;x) = n^{-1/2}\sum_{r=k}^\ell \Delta^m_\sigma b_n(x,Y^{(n)}_r)\;.
\]
Denote in the sequel $s=k/n$ and $t=\ell/n$.
Proposition~\ref{prop:moments_discrete} implies that for each $m\geq 1$ (cf.~\cite[p. 4088]{KM17})
\begin{equation}\label{eq:increment_bound_discrete_time}
\E_{\lambda_n}\left[|\Delta^m_\sigma  W_n(s,t; x)|^{2q}\right]^{1/(2q)} \lesssim |\Delta_\sigma^m b_n(x,\cdot)|_{C^\kappa}|t-s|^{1/2}\;.
\end{equation}
Setting $m=\roof{\theta+ \frac{d}{2q}}+1$, it follows that
\begin{align*}
\E_{\lambda_n}\Big[|W_n(s,t;\cdot)|_{\MB}^{2q}\Big]
&\lesssim
\E_{\lambda_n}\Big[|W_n(s,t;\cdot)|_{B^{\theta+d/(2q)}_{2q}}^{2q}\Big]
\\
&=
\E_{\lambda_n}\Big[ \int |W_n(s,t;x)|^{2q}\mrd x 
\\
&\qquad + \int_{|\sigma\leq 1|} |\sigma|^{-2\theta q - 2d}\int|\Delta^m_\sigma W_n(s,t;x)|^{2q} \mrd x \mrd\sigma \Big]
\\
&\lesssim
\int |b_n(x,\cdot)|_{C^\kappa}^{2q}|t-s|^{q}\mrd x \\
&\qquad + \int_{|\sigma| \leq 1} \int |\Delta^m_\sigma b_n(x,\cdot)|_{C^\kappa}^{2q}|t-s|^{q} \mrd x \mrd \sigma
\\
&\lesssim |t-s|^{q}|b_n|_{B^{\theta+d/(2q)}_{2q};C^\kappa}^{2q}
\lesssim |t-s|^{q}\;,
\end{align*}
where the first estimate follows from the embedding $B^{\theta+d/(2q)}_{2q} \hookrightarrow C^{\theta}$, the second from~\eqref{eq:increment_bound_discrete_time}, the third from the definition of $|\cdot |_{B^{\theta+d/(2q)}_{2q};C^\kappa}$~\cite[p. 4086]{KM17}, and the fourth from~\cite[Lemma~5.5]{KM17} since $b_n$ has uniformly bounded support and $\sup_{n \geq 1} |b_n|_{C^{\alpha,\kappa}} <\infty$ with $\alpha > \theta+d/(2q)$.

The second estimate follows in a similar way from Proposition~\ref{prop:moments_discrete} upon using the bound
\begin{equation*}
\E_{\lambda_n}\left[|\Delta^m_{x,\sigma} \Delta^{m'}_{x',\sigma'}  \BBW_n(s,t; x,x')|^{q}\right]^{1/q} \lesssim |\Delta_\sigma^m b_n(x,\cdot)|_{C^\kappa}|\Delta_{\sigma'}^{m'} b_n(x',\cdot)|_{C^\kappa}|t-s|
\end{equation*}
and the argument from~\cite[p. 4089-4090]{KM17} (note that this is where we require $\sup_{n \geq 1}|b_n|_{C^{\alpha,\kappa}} < \infty$ for $\alpha > \theta + d/q$, so that $\sup_{n \geq 1}|b_n|_{B^{\theta+d/q}_q} < \infty$).
\end{proof}

\begin{lem}[Tightness]
\label{lem:discrete_rp_tightness}
Suppose that Assumption~\ref{assump:discrete_W_assump}\ref{point:Kolm_assump} holds and that the support of $b_n$ is uniformly bounded.
Then, for any $p > 2$, it holds that
\[
\sup_{n\geq 1} \E_{\lambda_n}[\|\BW_n\|_{\var p}^{2q}] < \infty
\]
and that $(\BW_n)_{n \geq 1}$ is a family of tight random variables in $\MD^{\var p}(\MB)$.
\end{lem}

\begin{proof}
Consider $\theta' \in (\theta, \alpha-\frac{d}{q})$ and the space $\MB' = C^{\theta'}(\mathfrak{K},\R^d)$, where $\mathfrak{K}\times \Lambda$ contains the support of all $b_n$.
We first show, using a similar argument as in the proof of Lemma~\ref{lem:V_tight_discrete}, that $(\BW_n)_{n \geq 1}$ is tight in the ($J_1$) Skorokhod space $D([0,1], \MB\oplus\MB^{\otimes 2})$.
Indeed, Lemma~\ref{lem:unif_W_bound_discrete} and~\eqref{eq:piecewise_const_Hol_bound}
imply that $(\BW_n)_{n \geq 1}$ satisfies condition~\cite[Thm.~3.7.2(b)]{EK05},
and, since the embedding $\MB'\hookrightarrow \MB$ is compact, that
$(\BW_n)_{n \geq 1}$ satisfies the compact containment condition~\cite[Rem.~3.7.3]{EK05}.
Hence, by~\cite[Thm.~3.7.6]{EK05},
$(\BW_n)_{n \geq 1}$ is tight in $D([0,1], \MB\oplus\MB^{\otimes 2})$.

Consider now $p'\in (2,p)$.
The interpolation estimate~\eqref{eq:interpolation} implies
$\sigma_{\var{p}}(\BX,\BY) \leq \sigma_\infty(\BX,\BY)^{1-p'/p}(\|\BX\|_{\var{p'}}+\|\BY\|_{\var{p'}})^{p'/p}$
with $\sigma_\infty$ defined as in the proof of Lemma~\ref{lem:V_tight_discrete}.
Moreover, the map $\BX\mapsto \|\BX\|_{\var{p'}}$ is invariant under reparametrizations and is lower semi-continuous under $\|\cdot;\cdot\|_\infty$ and thus under $\sigma_\infty$.
Hence, for $R>0$ and a compact subset $K$ of the ($J_1$) Skorokhod space $D([0,1], \MB\oplus\MB^{\otimes 2})$, the set
\[
\{\BX \in K \ssep \|\BX\|_{\var{p'}} \leq R\}
\]
is compact in $\MD^{\var p}(\MB)$.
Considering $\BW_n$ as an element of $\MD^{\var{p'}}(\MB')$, it follows from Lemma~\ref{lem:unif_W_bound_discrete} and~\eqref{eq:piecewise_const_var_bound} (applied to these new parameters) that $\sup_{n \geq 1} \E_{\lambda_n}[\|\BW_n\|_{\var{p'}}^{2q}] < \infty$.
Consequently
$(\BW_n)_{n \geq 1}$ is tight in $\MD^{\var p}(\MB)$.
\end{proof}

For an element $\pi \in \BL(\MB,\R^m)$ and $b \in C^{\theta,\kappa}(\R^d\times \Lambda, \R^d)$, write $\pi b : M \to \R^m$ for the function $y \mapsto \pi(b(\cdot,y))$.
A direct verification shows that $|\pi b|_{C^\kappa} \leq |\pi|_{\BL(\MB,\R^m)}|b|_{C^{\theta,\kappa}}$ (see, e.g., the proof of~\cite[Lem.~5.12]{KM17}).

Consider the subspace of $\BL(\MB,\R)$
\[
\tBL(\MB,\R) = \operatorname{span} \big\{b \mapsto D^k b^j(x) \ssep x \in\R^d, \; k\in\BBN^d, |k|\leq 1,\; j \in \{1,\ldots, d\} \big\}\;.
\]
For $m \geq 1$, we denote by $\tBL(\MB,\R^m)$ the subspace of $\pi \in \BL(\MB,\R^m)$ such that $\pi^i\in\tBL(\MB,\R)$ for every $i=1,\ldots, m$.
We note that $\tBL(\MB,\R)$ does not appear in the work~\cite{KM17}, however, due to the generality of our setting, we find it more convenient to work with than the full space $\BL(\MB,\R)$.

Observe that, for $b \in C_n^{\theta,\kappa}(\R^d\times \Lambda, \R^d)$, the map $x \mapsto b(x,\cdot)$ is a $C^\theta$ map from $\R^d$ into the closed subspace $C^\kappa_n(\Lambda)$, and thus $\pi b \in C^\kappa_n(\Lambda)$ for all $\pi \in \tBL(\MB,\R)$.

\begin{lem}[Finite-dimensional projections]
\label{lem:discrete_finite_dim_proj}
Let $\pi \in \tBL(\MB,\R^m)$ for some $m \geq 1$ and suppose that Assumption~\ref{assump:discrete_W_assump} holds.
Let $\mfB$ be defined as in Proposition~\ref{prop:auto}, and let $W_\pi$ be an $\R^m$-valued Brownian motion with covariance
\[
\E[W^i_\pi(1)W^j_\pi(1)] = \mfB(\pi^i b_\infty, \pi^j b_\infty) + \mfB(\pi^j b_\infty, \pi^i b_\infty)\;.
\]
Define further
\[
\BBW^{i,j}_\pi(t) = \int_0^t W^i_\pi \mrd W_\pi^j + \mfB_2(\pi^i b_\infty, \pi^j b_\infty)t\;.
\]
Then, as $n \to \infty$,
\[
(\pi W_n, (\pi\otimes\pi){\BBW_n}) \to_{\lambda_n} (W_\pi,\BBW_\pi)
\]
in the sense of finite-dimensional distributions.
\end{lem}

\begin{proof}
By the preceding remarks, 
$\pi b_n\in C^\kappa_n(\Lambda,\R^m)$ for $n\in\BBN\cup\{\infty\}$ and 
$\lim_{n \to \infty} |\pi b_n-\pi b_\infty|_{C^\kappa} = 0$,
%$(\pi b_n)_{n\in\BBN\cup\{\infty\}}$ is a $C^\kappa_n(\Lambda,\R^m)$-family, 
so the conclusion follows by Assumption~\ref{assump:discrete_W_assump} and Proposition~\ref{prop:auto}.
\end{proof}

The convergence of finite-dimensional distributions, together with tightness, allows us to establish uniqueness of weak limit points (which we note settles a point of ambiguity in~\cite[Rem.~5.14]{KM17}).

\begin{prop}\label{prop:discrete_unique_limit}
Suppose that Assumptions~\ref{assump:conv_drift} and~\ref{assump:discrete_W_assump} hold
and that the support of $a_n$ and $b_n$ is uniformly bounded.
Let $\beta>1$ and $p>2$.
Then there exists a random variable $(V,\BW)$ in $D^{\var\beta}(\MA)\times \MD^{\var p}(\MB)$ such that $(V_n,\BW_n) \to_{\lambda_n} (V,\BW)$.
Furthermore, $(V,\BW)$ is a.s. continuous.
\end{prop}

\begin{proof}
By Lemma~\ref{lem:V_tight_discrete}, $\|V_n - V\|_{\var\beta} \to_{\lambda_n} 0$, where $V$ is deterministic and continuous.
It remains to show that $\BW_n$ converges weakly to a limit point $\BW$ which is a.s. continuous.
By Lemma~\ref{lem:discrete_rp_tightness}, $(\BW_n)_{n \geq 1}$ is tight
and thus weakly relatively compact by Prokhorov's theorem~\cite[Thm.~5.1]{Billingsley}.
Let $\BW$ and $\tilde \BW$ be weak limit points of two subsequences of $\BW_n$.
Since the largest jump of $W_n$ is of the order $n^{-1/2}$ and the largest jump of $t \mapsto \BBW_n(0,t)$ is of the order $n^{-1/2}\sup_{t\in[0,1]}|W_n(t)|_{\MB}$, it follows that $\BW$ is a.s continuous (and likewise for $\tBW$).

We now show that $\BW$ and $\tBW$ have the same law.
Consider the collection of $\R$-valued functions on $\MD^{\var p}(\MB)$
\begin{align*}
\MF := \Big\{ (X,\X) \mapsto \sum_{j=1}^k\tau_j(\pi_j X(t_j), (\pi_j\otimes\pi_j)\X(t_j))\Big\}\;,
\end{align*}
where the parameters range over all $k \geq 1$, $\tau_j \in \BL(\R^{m_j}\oplus(\R^{m_j})^{\otimes 2},\R)$, $\pi_j \in \tBL(\MB,\R^{m_j})$, $m_j\geq 1$, and $t_j \in [0,1]$.
For any $f \in \MF$, it follows from Lemma~\ref{lem:discrete_finite_dim_proj} that $f(\BW)$ and $f(\tilde \BW)$ have the same law.
In particular, $\E[e^{if(\BW)}] = \E[e^{if(\tilde\BW)}]$ for all $f\in\MF$.
However, the collection of $\C$-valued functions $\tilde \MF := \{ w \mapsto e^{i f(w)} \ssep f \in \MF\}$ is a unital algebra of bounded functions on $\MD^{\var p}(\MB)$ which separates points and is closed under conjugation.
Moreover, every $f\in\tilde \MF$ is continuous on the subspace of continuous paths in $\MD^{\var p}(\MB)$, and in particular on the support of $\BW$ and $\tilde \BW$.
The laws of $\BW$ and $\tilde \BW$ are Radon measures since they are obtained as weak limit points of tight sequences,
hence, by the Stone--Weierstrass theorem and a compactification argument (see, e.g.,~\cite[Ex.~7.14.79]{Bogachev07}),
$\BW$ and $\tilde \BW$ have the same law.
\end{proof}

It remains to characterize the RDE driven by $(V,\BW)$ as the solution to an SDE.
We flesh out the abstract statement in the following lemma, which is a slight simplification of~\cite[Lem.~6.1]{KM17}.

\begin{lem}\label{lem:RDE_to_SDE}
Let $X$ be the solution to the RDE
\[
\mrd X = F(X) \bar a \mrd t + H(X) \mrd \BW\;, \quad X(0)=\xi \in \R^d\;,
\]
where $\bar a \in \MA$ is fixed and $\BW=(W,\BBW)$ is a random $p$-rough path over $\MB$, $p < \theta$.
Suppose that, for all $m \geq 1$ and $\pi\in\tBL(\MB,\R^m)$,
\begin{equation}\label{eq:fdd_condition}
(\pi W, (\pi\otimes \pi)\BBW) \sim (W_\pi,\BBW_\pi)
\end{equation}
in the sense of finite dimensional distributions,
where $W_\pi$ is an $\R^m$-valued Brownian motion with covariance
\[
\Sigma_\pi^{ij} := \E[W^i_\pi(1) W_\pi^j(1)]
\]
and
\[
\BBW_\pi^{ij}(t) = \int_0^t W^i_\pi(s) \mrd W^j_\pi(s) + \Gamma_\pi^{ij} t\;.
\]
For every $x \in \R^d$, let us define $\Sigma(x) := \Sigma_{H(x)}$
and, for $i = 1,\ldots, d$,
\[
\Gamma^i(x) := \sum_{k=1}^d\Gamma^{k(ki)}_{H(x)\oplus DH(x)}\;,
\]
where we treat $H(x)\oplus DH(x) \in \tBL(\MB,\R^d\oplus (\R^d)^*\otimes \R^d)$.
Suppose further that
\begin{equation}\label{eq:Sigma_Gamma_bound}
\sup_{x\in\R^d} \sum_{i=1}^d|\Sigma^{ii}(x)| + |\Gamma^i(x)|  < \infty\;.
\end{equation}
Then $X$ solves the martingale problem associated with $\ML = (\bar a + \Gamma)D + \frac12 \Sigma D^2$.
\end{lem}

\begin{proof}
Let $\{\MF_t\}_{t\in[0,1]}$ denote the filtration generated by the finite-dimensional projections of $\BW$.
We first show that $M : [0,1] \to \R^d$ is a martingale with respect to $\MF$, where
\[
M(t) := X(t) - \int_0^t \bar a (X(s)) \mrd s - \int_0^t \sum_{k=1}^d \Gamma(X(s)) \mrd s\;,
\]
with quadratic variation
\begin{equation}\label{eq:M_quad_var}
[M^i,M^j]_t = \int_0^t \Sigma^{ij}(X(s)) \mrd s\;.
\end{equation}
Indeed, the definition of the rough integral readily implies that $X$ and $M$ are adapted to $\MF$ (cf.~\cite[Lem~6.3]{KM17}).
Furthermore, for fixed $0\leq s < t \leq 1$, we have
\begin{align*}
M(t)-M(s)
&= \int_s^t H(X(u)) \mrd \BW(u) - \int_s^t \Gamma(X(u)) \mrd u
\\
&= \lim_{|\MP|\to 0} \sum_{[u,v]\in\MP} M^{\MP}_{[u,v]}\;,
\end{align*}
where the limit is taken over partitions $\MP$ of $[s,t]$, and
\[
M^{\MP}_{[u,v]} := H(X(u))W(u,v) + (H\otimes DH)(X(u))\BBW(u,v) - \Gamma(X(u))(v-u)\;.
\]
Note that the same argument as in~\cite[Lem.~6.2]{KM17} implies that $\pi W(u,v)$ and $(\pi\otimes\pi)\BBW(u,v)$ are independent of $\MF_s$ for any $\pi\in\tBL(\MB,\R^m)$.
Taking $\pi(x) = H(x)\oplus DH(x)$ in~\eqref{eq:fdd_condition}, it follows that
\[
\E[M^\MP_{[u,v]}\mid \MF_u] = 0\;.
\]
Furthermore, for $i,j=1,\ldots, d$,
\[
\E[H^i(X(u))W(u,v) H^j(X(u))W(u,v) \mid \MF_u] = \Sigma^{ij}(X(u))(v-u)
\]
and, by It{\^o} isometry,
\[
\E[|(H\otimes DH)(X(u))\BBW(u,v) - \Gamma(X(u))(v-u)|^2 \mid \MF_u] \lesssim |v-u|^2\;,
\]
where the proportionality constant depends only on $\Sigma(X(u))$.
Using the bound~\eqref{eq:Sigma_Gamma_bound}, it follows that $M$ is a martingale with quadratic variation~\eqref{eq:M_quad_var} as claimed.

Let $\varphi : \R^d \to \R$ be a smooth, compactly supported function.
Since $[X]=[M]$, by It{\^o}'s formula,
\begin{align*}
\varphi(X(t))
= \varphi(X(s)) + \int_s^t D\varphi(X(u)) \mrd X(u) + \frac12\int_s^t D^2\varphi(X(u))\mrd [M](u)\;,
\end{align*}
from which it follows that
\[
\varphi(X(t)) - \varphi(X(s)) - \int_s^t \Big[ D\varphi (\bar a + \Gamma) + \frac12 D^2\varphi \Sigma\Big](X(u)) \mrd u
\]
is a martingale.
\end{proof}

\begin{proof}[Proof of Theorem~\ref{thm:RP_discrete_conv}]
The fact that $(V_n,\BW_n) \to_{\lambda_n} (V,\BW)$, where $(V,\BW)$ is a.s. continuous, follows from Proposition~\ref{prop:discrete_unique_limit}.
By Lemma~\ref{lem:discrete_finite_dim_proj}, $\BW$ satisfies assumption~\eqref{eq:fdd_condition} of Lemma~\ref{lem:RDE_to_SDE} with $\Sigma^{ij}_\pi = \mfB(\pi^ib,\pi^jb) + \mfB(\pi^jb,\pi^ib)$ and $\Gamma_\pi^{ij} = \mfB_2(\pi^i b,\pi^jb)$.
In particular, $\Gamma$ in Lemma~\ref{lem:RDE_to_SDE} is given by $\Gamma^i(x) = \sum_{k=1}^d \mfB_2(b^k(x,\cdot),\partial_kb^i(x,\cdot))$.
Furthermore, $\mfB=\frac12\mfB_1+\mfB_2$ is bounded by 
Assumption~\ref{assump:discrete_W_assump}\ref{point:fdd_assump}, so
$\Sigma^{ii}(x) \lesssim |b(x,\cdot)|_{C^\kappa}^2 \leq |b|_{C^{0,\kappa}}^2$ and $\Gamma^i(x) \lesssim |b(x,\cdot)|_{C^\kappa}|\nabla b (x,\cdot)|_{C^\kappa} \leq |b|_{C^{0,\kappa}}|b|_{C^{1,\kappa}}$.
Hence all the assumptions of Lemma~\ref{lem:RDE_to_SDE} are verified, and the conclusion follows from~\cite[Thm.~4.5.2]{StroockVaradhan06} by the equivalence of weak solutions to SDEs and the martingale problem.
\end{proof}

\begin{proof}[Proof of Theorem~\ref{thm:discrete_FS_conv}]
This follows from Corollary~\ref{cor:conv_b_compact_supp} and the exact same localization argument as in~\cite[Sec.~7]{KM17}.
\end{proof}

\section{Continuous-time dynamics revisited}
\label{sec:cont_fs}

In this section, we show how the results of the Section~\ref{sec:discrete_fs} extend to the case of continuous-time dynamics.
In particular, we extend the results of~\cite{KM17} to include optimal moment assumptions and families of dynamical systems.
Since the arguments are very similar to those of the discrete-time case (and the setting is similar to that of~\cite{KM17}), we omit the proofs and only state the main results.

Consider a compact Riemannian manifold $M$ with Riemannian distance $\rho$.
Recall the function spaces defined in Definition~\ref{def:discrete_setup} and fix parameters $q>1$, $\kappa,\bar\kappa \in (0,1)$, and $\alpha > 2+\frac dq$.
Let $a_\eps\in C^{1+\bar\kappa,0}(\R^d\times M,\R^d)$ and $b_\eps, b_0 \in C^{\alpha,\kappa}(\R^d\times M,\R^d)$, for $\eps \in (0,1]$, such that
\[
\sup_{\eps\in(0,1]} |a_\eps|_{C^{1+\bar\kappa,0}} + |b_\eps|_{C^{\alpha,\kappa}} < \infty\;, \quad \lim_{\eps \to 0} |b_\eps - b_0|_{C^{\alpha,\kappa}} = 0\;.
\]
We consider the fast-slow systems of ODEs posed on $\R^d\times M$
\[
\frac{d}{dt} x_\eps =a_\eps(x_\eps,y_\eps) +\eps^{-1}b_\eps(x_\eps,y_\eps)\;, \qquad
\frac{d}{dt} y_\eps =\eps^{ -2}g_\eps(y_\eps)\;,
\]
where $g_\eps:M\to TM$ is a Lipschitz vector field.
As before, the initial condition $x_\eps(0) = \xi_\eps\in\R^d$ is deterministic, and $y_\eps(0)$ is drawn randomly from a Borel probability measure $\lambda_\eps$ on $M$.

We now give the analogues of Assumptions~\ref{assump:conv_drift} and~\ref{assump:discrete_W_assump} for the current setting.

\begin{assumption}\label{assump:cont_conv_drift}
There exists $\bar a \in C^{1+\bar\kappa}(\R^d,\R^d$) such that, for all $t\in[0,1]$ and $x\in\R^d$,
\[
|V_\eps(t)(x)-t\bar a(x)|\to_{\lambda_\eps}0 \quad \textnormal{ as } \quad \eps\to 0\;,
\]
where $V_\eps(t) = \int_0^t a_\eps(\cdot,y_\eps(s))\mrd s$.
\end{assumption}

Let $g_{\eps,t}$ denote the flow generated by the vector field $g_\eps$.
Given $v,w \in C^\kappa(M,\R^m)$ and $0\le s\le t\le 1$, we define
$W_{v,\eps}(t)\in \R^m$ and
$\BBW_{v,w,\eps}(s,t)\in \R^{m\times m}$ by
\begin{align*}
W_{v,\eps}(t)=\eps\int_0^{t\eps^{-2}}v\circ g_{\eps,s}\mrd s\,,\quad 
\BBW_{v,w,\eps}(s,t)=\int_s^t (W_{v,\eps}(r)-W_{v,\eps}(s)) \otimes \mrd W_{w,\eps}(r)\;.
\end{align*}
As before, we write simply $\BBW_{v,\eps}$ for $\BBW_{v,v,\eps}$.

Recall our notational convention about subspaces $C^\kappa_\eps(M)$ of $C^\kappa(M)$ introduced before Assumption~\ref{assump:discrete_W_assump_n_indep}.
%As in Section~\ref{subsec:general_case}, given a family of subspaces $(C^\kappa_\eps(M))_{\eps\in[0,1]}$ of $C^\kappa(M)$, we call $\bv = (v_\eps)_{\eps\in[0,1]}$ a $C^\kappa_\eps(M,\R^m)$-family if $v_\eps \in C^\kappa_\eps(M,\R^m)$ and $\lim_{\eps\to 0} |v_\eps - v_0|_{C^\kappa} = 0$.

\begin{assumption}\label{assump:cont_W_assump}
There exists a closed subspace $C_\eps^\kappa(M)$ of $C^\kappa(M)$ for each $\eps \in [0,1]$ such that $b_\eps \in C_\eps^{\alpha,\kappa}(\R^d\times M,\R^d)$ and such that
\begin{enumerate}[label=(\roman*)]
\item\label{point:Kolm_assump_cont} for all $v=(v_\eps), w = (w_\eps) \in \prod_{\eps\in(0,1]} C_\eps^\kappa(M)$ with
\[
\sup_{\eps\in(0,1]} |v_\eps|_{C^\kappa} + |w_\eps|_{C^\kappa} < \infty\;,
\]
there exists
$K = K_{v,w,q}>0$ such that for all $0\leq s\leq t \leq 1$ and $\eps > 0$,
\[
|W_{v_\eps,\eps}(s,t)|_{L^{2q}(\lambda_\eps)}\le K |t-s|^{1/2}\;, \quad
|\BBW_{v_\eps,w_\eps,\eps}(s,t)|_{L^{q}(\lambda_\eps)}\le K |t-s|\;.
\]

\item\label{point:fdd_assump_cont} There exists a bounded bilinear operator $\mfB:C_0^\kappa(M)\times C_0^\kappa(M) \to\R$ such that for every $m\ge1$ and all  $\bv=(v_\eps)_{\eps\in[0,1]}$ with $v_\eps\in C^\kappa_\eps(M,\R^m)$ and
$\lim_{\eps\to 0} |v_\eps - v_0|_{C^\kappa} = 0$,
%every $C^\kappa_\eps(M,\R^m)$-family $\bv=(v_\eps)_{\eps\in[0,1]}$,
it holds that $(W_{v_\eps,\eps},\BBW_{v_\eps,\eps})\to_{\lambda_\eps} (W_\bv,\BBW_\bv)$ as $\eps\to0$
in the sense of finite-dimensional distributions, where $W_{\bv}$ is an $\R^m$-valued Brownian motion and
\[
\BBW^{ij}_{\bv}(t)=\int_0^t W_{\bv}^i\mrd W_{\bv}^j+\mfB(v_0^i,v_0^j)t\;.
\]
\end{enumerate}  
\end{assumption}

\begin{rem}
As in Remark~\ref{rem:stationary_simplification},
under the assumption that $\lambda_\eps$ is $g_{\eps,t}$-stationary, the simpler bounds
\[
|W_{v_\eps,\eps}(t)|_{L^{2q}(\lambda_\eps)} \le K t^{1/2} 
\quad\text{and} \quad
|\BBW_{v_\eps,w_\eps,\eps}(0,t)|_{L^{q}(\lambda_\eps)} \le K t
\quad\text{for all $\eps, t \in (0,1]$}
\]
imply Assumption~\ref{assump:cont_W_assump}\ref{point:Kolm_assump_cont}.
\end{rem}

\begin{rem}
As in Proposition~\ref{prop:auto}, one can show that Assumption~\ref{assump:cont_W_assump} implies that the covariance of $W_{\bv}$ is given by
\[
\E[W^i_{\bv}(1)W^j_{\bv}(1)] = \mfB(v_0^i,v_0^j) + \mfB(v_0^j,v_0^i)\;.
\]
Furthermore, as in Section~\ref{subsec:n_indep}, if $a_\eps,b_\eps,T_\eps,\lambda_\eps$ do not depend on $\eps$, then one can drop the condition that $\mfB$ is bounded in Assumption~\ref{assump:cont_W_assump} since this follows automatically (see~\cite[Prop.~2.8]{KM17}).
\end{rem}

Consider the quadratic form
\begin{equation}\label{eq:Sigma_def_cont}
\Sigma^{ij}(x) = \mfB(b^i_0(x,\cdot),b^j_0(x,\cdot)) + \mfB(b^j_0(x,\cdot),b^i_0(x,\cdot))\;, \quad i,j =1, \dots, d\;.
\end{equation}
By the same argument as Lemma~\ref{lem:Lip_sqrt_n_indep}, $\Sigma$ is positive semi-definite and the unique positive semi-definite $\sigma$ satisfying $\sigma^2 = \Sigma$ is Lipschitz.
In particular, as before, there is a unique (strong) solution to the SDE $\mrd X = \tilde a(X)\mrd t + \sigma(X)\mrd B$ for any Lipschitz $\tilde a : \R^d\to\R^d$.

The following is the main result of this section, the proof of which we omit since it requires only minor changes to that of Theorem~\ref{thm:discrete_FS_conv}.

\begin{thm}\label{thm:cont_FS_conv}
Suppose that Assumptions~\ref{assump:cont_conv_drift} and~\ref{assump:cont_W_assump} hold, and that $\xi_\eps \to \xi \in \R^d$.
Then $x_\eps \to_{\lambda_\eps} X$ in the uniform topology as $\eps \to 0$, where $X$ is the unique weak solution of the SDE
\begin{equation}\label{eq:nonprod sde}
\mrd X = \tilde a(X)\mrd t + \sigma(X)\mrd B\;, \quad X(0) = \xi\;.
\end{equation}
Here, $B$ is a standard Brownian motion in $\R^d$, $\sigma$ is the unique positive semi-definite square root of $\Sigma$ given by~\eqref{eq:Sigma_def_cont}, and $\tilde a$ is the Lipschitz function given by
\begin{equation*}
\tilde a^i(x) = \bar a^i(x) +  \sum_{k=1}^d \mfB(b^k_0(x,\cdot),\partial_k b^i_0(x,\cdot))\;,
\quad i=1,\dots,d\;.
\end{equation*}
\end{thm}

\appendix

\section{Rough path Besov-variation embedding}
\label{appendix:Besov}

We adapt Friz--Victoir~\cite{FV06, FV10} in proving some variants of a Besov-variation embedding, applicable in an infinite-dimensional rough path setting.
Let $\MB$ be a Banach space and equip $\MB^{\otimes 2},\ldots, \MB^{\otimes N}$ with a system of admissible tensor norms.
For a continuous multiplicative function $\BW = (1,\BW^1,\ldots, \BW^N) : [0,T]^2 \to \oplus_{k=0}^N \MB^{\otimes k}$ define the homogeneous Besov norm
\[
\|\BW\|_{W^{\alpha,q};[s,t]}^q := \sum_{k=1}^N \iint_{[s,t]^2} \frac{|\BW^k_{v,u}|_{\MB^{\otimes k}}^{q/k}}{|u-v|^{q\alpha+1}} \mrd u \mrd v\;.
\]
\begin{prop}\label{prop:Besov to Holder}
Suppose $q>1$ and $\alpha \in (\frac{1}{q},1)$.
There exists a constant $C = C(\alpha,q,N)$ such that
\[
\sum_{k=1}^N |\BW^k_{s,t}|^{q/k} \leq C |t-s|^{q\alpha-1} \|\BW\|_{W^{\alpha,q};[s,t]}^q\;.
\]
\end{prop}
\begin{proof}
We follow a similar strategy to~\cite[Proposition~A.9]{FV10}.
We proceed by induction on $N$.
The case $N=1$ follows directly from the GRR lemma~\cite[Corollary~A.2]{FV10}.
Suppose the result is true for $N-1$.
Since both sides scale homogeneously with dilations, we may suppose that $\|\BW\|_{W^{\alpha,q};[s,t]}^q\leq 1$.
Let us write $\alpha-\frac{1}{q} =: 1/p$.
All double integrals in the sequel are taken over $[s,t]^2$, and $C$ denotes an unimportant positive constant which may change from line to line.

Define $\Upsilon_{s,t} = \sup_{u,v\in[s,t]}\frac{|\BW_{u,v}|}{|v-u|^{N/p}}$, and observe that it suffices to show $\Upsilon_{s,t} \leq C$.
We have
\begin{align*}
\BW^N_{s,v} - \BW^N_{s,u} = \BW^N_{u,v} + \sum_{j=1}^{N-1} \BW^{N-j}_{s,u}\otimes\BW^j_{u,v}\;,
\end{align*}
and thus
\[
\Big(\iint\frac{|\BW^N_{s,u} - \BW^N_{s,v}|^{q}}{|v-u|^{q\alpha+1}}\mrd u \mrd v\Big)^{1/q} \leq \Delta_1 + \Delta_2\;,
\]
where
\begin{align*}
\Delta_1
&= \sum_{j=1}^{N-1}\Big(\iint |\BW^{N-j}_{s,u}|^{q}\frac{|\BW^j_{u,v}|^q}{|u-v|^{q\alpha+1}}\mrd u \mrd v \Big)^{1/q}\;,
\\
\Delta_2
&= \Big(\iint\frac{|\BW^N_{u,v}|^q}{|u-v|^{q\alpha+1}}\mrd u \mrd v \Big)^{1/q}\;.
\end{align*}
For $\Delta_1$, by the inductive hypothesis, we have $|\BW^{(N-j)}_{s,u}|^q \leq |t-s|^{q(N-j)/p}$, so that
\[
\Delta_1 \leq \sum_{j=1}^{N-1} |t-s|^{(N-j)/p}\Big(\iint \frac{|\BW^j_{u,v}|^q}{|u-v|^{q\alpha+1}}\mrd u \mrd v\Big)\;.
\]
Again by the inductive hypothesis, we have
\begin{align*}
|\BW^j_{u,v}|^{q(1-1/j)} \leq |t-s|^{q(j-1)/p}\;,
\end{align*}
so that
\begin{align*}
\Delta_1 \leq \sum_{j=1}^N |t-s|^{(N-1)/p}\Big(\iint \frac{|\BW^j_{u,v}|^{q/j}}{|u-v|^{q\alpha+1}}\mrd u \mrd v\Big) \leq \sum_{j=1}^N |t-s|^{(N-1)/p}\;.
\end{align*}
For $\Delta_2$, we have
\begin{align*}
\Delta_2
&\leq \Big( \iint \Upsilon_{s,t}^{q(1-1/N)}|t-s|^{q(N-1)/p}\frac{|\BW^N_{u,v}|^{q/N}}{|v-u|^{q\alpha+1}}\mrd u \mrd v \Big)^{1/q}
\\
&\leq \Upsilon_{s,t}^{1-1/N}|t-s|^{(N-1)/p}\;.
\end{align*}
Combining the above two estimates, we have
\begin{align*}
\Big(\iint\frac{|\BW_{s,u}^N - \BW^N_{s,v}|^{q}}{|v-u|^{q\alpha+1}}\mrd u \mrd v\Big)^{1/q} \leq C |t-s|^{(N-1)/p}(1 + \Upsilon_{s,t}^{1-1/N})\;.
\end{align*}
Applying the GRR lemma to the continuous path $\BW^N_{s,\cdot} : [s,t] \to \MB^{\otimes N}$ we have
\begin{align*}
|\BW^N_{s,t}|
&\leq C|t-s|^{1/p}|t-s|^{(N-1)/p}(1+\Upsilon_{s,t}^{1-1/N})
\\
&\leq C |t-s|^{N/p}(1+\Upsilon_{s,t}^{1-1/N})\;.
\end{align*}
Finally, note that the above argument applies to any interval $[s',t'] \subset [s,t]$.
It follows that
\[
\Upsilon_{s,t} \leq C(1+\Upsilon^{1-1/N}_{s,t})\;,
\]
and thus $\Upsilon_{s,t} \leq C$ as desired.
\end{proof}

Recall the homogeneous $\gamma$-H{\"o}lder ``norm'' for $\gamma \in (0,1]$
\[
\|\BW\|_{\Hol\gamma;[s,t]} := \sum_{k=1}^N \sup_{u,t\in[s,t]} \frac{|\BW^k_{v,u}|^{1/k}}{|u-v|^{\gamma}}\;.
\]
\begin{cor}\label{cor:BesovHolderEmbedding}
Let $q>1$ and $\alpha \in (\frac{1}{q},1)$.
There exists a constant $C = C(\alpha,q,N)$ such that
\[
\|\BW\|_{\Hol{(\alpha-1/q)};[s,t]} \leq C\|\BW\|_{W^{\alpha,q};[s,t]}\;.
\]
\end{cor}
\begin{proof}
Immediate from Proposition~\ref{prop:Besov to Holder}.
\end{proof}

Recall the homogeneous $p$-variation ``norm'' for $p \geq 1$
\[
\|\BW\|^p_{\var{p};[s,t]} := \sup_{\op} \sum_{[u,v]\in\op} \sum_{k=1}^N|\BW^k_{u,v}|^{p/k}\;,
\]
where the supremum runs over all partitions $\op$ of $[s,t]$.

\begin{cor}\label{cor:BesovVarEmbedding}
Let $q>1$ and $\alpha \in (\frac{1}{q},1)$.
There exists a constant $C = C(\alpha,q,N)$ such that
\[
\|\BW\|_{\var{1/\alpha};[s,t]} \leq C|t-s|^{\alpha-1/q}\|\BW\|_{W^{\alpha,q};[s,t]}\;.
\]
\end{cor}

\begin{proof}
By Proposition~\ref{prop:Besov to Holder} we have for all $u,v \in [s,t]$ and $k=1,\ldots, N$
\begin{align*}
|\BW^k_{u,v}|^{\frac{1}{\alpha k}}
&= \Big(|\BW^k_{u,v}|^{q/k}\Big)^{\frac{1}{\alpha q}}
\\
&\leq C \left(|u-v|^{q\alpha - 1}\right)^{\frac{1}{q\alpha}} \left(\|\BW\|^q_{W^{\alpha,q};[u,v]}\right)^{\frac{1}{q\alpha}}\;.
\end{align*}
Note however that $\omega_1(u,v) = |u-v|$ and $\omega_2(u,v) := \|\BW\|^q_{W^{\alpha,q};[u,v]}$ are controls, and thus so is $\omega := \omega_1^{1-\frac{1}{q\alpha}}\omega_2^{\frac{1}{q\alpha}}$.
Hence
\[
\|\BW\|^{1/\alpha}_{\var{1/\alpha};[s,t]} \leq \omega(s,t)\;,
\]
from which the conclusion follows.
\end{proof}

\begin{rem}
Besov (rough path) regularity effectively interpolates between the well-known H\"older- and $p$-variation cases, see~\cite{FP16x} for a discussion.
\end{rem}


\begin{thebibliography}{10}

\bibitem{BC17}
I.~Bailleul and R.~Catellier. Rough flows and homogenization in stochastic turbulence. \emph{J. Differential Equations} \textbf{263} (2017) 4894--4928.

\bibitem{Billingsley}
P.~Billingsley. \emph{Convergence of probability measures.} Second edition. Wiley Series in Probability and Statistics, John Wiley \& Sons, Inc., New York, 1999.


\bibitem{Bogachev07}
V.~Bogachev. \emph{Measure theory. Vol. I, II.} Springer-Verlag, 2007.   

  
\bibitem{Chevyrev18}
I.~Chevyrev. Random walks and L{\'e}vy processes as rough paths.
\emph{Probab. Theory Related Fields} \textbf{170} (2018) 891--932.

\bibitem{CF19} I.~Chevyrev and P.~K.~Friz.
Canonical RDEs and general semimartingales as rough paths.
\emph{Ann. Probab.} \textbf{47} (2019) 420--463.

\bibitem{CFKMZ19}
I.~Chevyrev, P.~K. Friz, A.~Korepanov, I.~Melbourne and H.~Zhang. Multiscale
  systems, homogenization, and rough paths. \emph{Probability and Analysis in
  Interacting Physical Systems: In Honor of S.R.S. Varadhan, Berlin, August,
  2016''} (P.~Friz et~al., ed.). Springer Proceedings in Mathematics \&
  Statistics \textbf{283}, Springer, 2019, pp.~17--48.

% \bibitem{CFKMZsub}
% I.~Chevyrev, P.~K. Friz, A.~Korepanov, I.~Melbourne and H.~Zhang.
% Deterministic homogenization under optimal moment assumptions for fast-slow systems.  Part 2.
% Preprint, 2019.

\bibitem{EK05}
S.~N.~Ethier and T.~G.~Kurtz.
\emph{Markov processes: Characterization and convergence.}
Wiley Series in Probability and Mathematical Statistics, John Wiley \& Sons, Inc., New York, 2005.


%\bibitem{FrizHairer}
%P.~K. Friz and M.~Hairer. \emph{A course on rough paths}. Universitext,
%  Springer, Cham, 2014.

\bibitem{FP16x}
P.~K. {Friz} and D. Pr{\"o}mel.
Rough path metrics on a Besov--Nikolskii-type scale.
\emph{Trans. Amer. Math. Soc.} \textbf{370} (2018) 8521--8550.
  
\bibitem{FS17}
P.~K. {Friz} and A. Shekhar.
General rough integration, L{\'e}vy rough paths and a L{\'e}vy--Kintchine-type formula.
\emph{Ann. Probab.} \textbf{45} (2017) 2707--2765.

\bibitem{FV06}
P.~K. {Friz} and N. Victoir.
A variation embedding theorem and applications.
\emph{J. Funct. Anal.} \textbf{239} (2006) 631--637.

\bibitem{FV10} P.~K.~Friz and N.~B.~Victoir.
\emph{Multidimensional stochastic processes as rough paths.}
Cambridge Stud. Adv. Math. \textbf{120}, 2010.

\bibitem{FZ18}
P.~K. {Friz} and H.~{Zhang}.
Differential equations driven by rough paths with jumps.
\emph{J. Differential Equations} \textbf{264} (2018) 6226--6301.

\bibitem{GM13b}
G.~A. Gottwald and I.~Melbourne. {Homogenization for deterministic maps and
  multiplicative noise}. \emph{Proc. R. Soc. London A} \textbf{469} (2013)
  20130201.

  
\bibitem{Kelly16}
D.~Kelly. Rough path recursions and diffusion approximations. \emph{Ann. Appl. Probab.} \textbf{26} (2016) 425--461.


\bibitem{KM16}
D.~Kelly and I.~Melbourne. {Smooth approximation of stochastic differential
  equations}. \emph{Ann. Probab.} \textbf{44} (2016) 479--520.

\bibitem{KM17}
D.~Kelly and I.~Melbourne. Homogenization for deterministic fast-slow systems
  with multidimensional multiplicative noise. \emph{J. Funct. Anal.}
  \textbf{272} (2017) 4063--4102.


\bibitem{KKM18}
A.~Korepanov, Z.~Kosloff and I.~Melbourne. Martingale-coboundary decomposition
  for families of dynamical systems. \emph{Ann. Inst. H. Poincar\'e Anal. Non
  Lin\'eaire} \textbf{35} (2018) 859--885.
  
\bibitem{KKMprep}
A.~Korepanov, Z.~Kosloff and I.~Melbourne.
Deterministic homogenization under optimal moment assumptions for fast-slow systems.  Part 1.
Preprint, 2020.

\bibitem{LSV99}
C.~Liverani, B.~Saussol and S.~Vaienti. {A probabilistic approach to
  intermittency}. \emph{Ergodic Theory Dynam. Systems} \textbf{19} (1999)
  671--685.
  
\bibitem{Lyons98}
T.~J. Lyons. Differential equations driven by rough signals. \emph{Rev. Mat.
  Iberoamericana} \textbf{14} (1998) 215--310.

\bibitem{MS11}
I.~Melbourne and A.~Stuart. {A note on diffusion limits of chaotic skew product
  flows}. \emph{Nonlinearity} \textbf{24} (2011) 1361--1367.

\bibitem{PomeauManneville80}
Y.~Pomeau and P.~Manneville. Intermittent transition to turbulence in
  dissipative dynamical systems. \emph{Comm. Math. Phys.} \textbf{74} (1980)
  189--197.
  


\bibitem{Ryan02} R.~Ryan.
\emph{Introduction to tensor products of Banach spaces.}
Springer Monographs in Mathematics, Springer-Verlag London, 2002.
  
\bibitem{Smale67}
S.~Smale. Differentiable dynamical systems. \emph{Bull. Amer. Math. Soc.} \textbf{73} (1967)
747--817.

\bibitem{StroockVaradhan06}
D.~Stroock and S.~Varadhan. \emph{Multidimensional diffusion processes}. Springer-Verlag, Berlin, 2006.

\bibitem{Young98}
L.-S. Young. Statistical properties of dynamical systems with some
  hyperbolicity. \emph{Ann. of Math.} \textbf{147} (1998) 585--650.

\bibitem{Young99}
L.-S. Young. Recurrence times and rates of mixing. \emph{Israel J. Math.}
  \textbf{110} (1999) 153--188.

\end{thebibliography}
\end{document}